\documentclass[11pt]{article}

\usepackage[a4paper,headinclude=false,footinclude=false,DIV=10]{typearea}
\usepackage{setspace}

\usepackage{amsmath,amsthm,amssymb,amscd}

\usepackage{epsfig,graphicx}
\usepackage[colorlinks=false]{hyperref}
\usepackage[english]{babel}

\usepackage[utf8]{inputenc}  
\usepackage[T1]{fontenc}
\usepackage{fourier} 
\usepackage[scaled=0.875]{helvet}

\newtheoremstyle{mystyle}{}{}{\rmfamily}%
{}{\normalfont\bfseries}{ }{ }{} 

\newtheorem{theorem}{Theorem}[section]
\newtheorem{proposition}[theorem]{Proposition}
\newtheorem{lemma}[theorem]{Lemma} 
\newtheorem{corollary}[theorem]{Corollary}
\newtheorem{definition}[theorem]{Definition}
\theoremstyle{mystyle}
\newtheorem{remark}[theorem]{Remark}
\newtheorem{example}{Example}

\newcommand{\R}{\mathbb{R}}
\newcommand{\N}{\mathbb{N}}

\newcommand{\Z}{\mathbb{Z}}
\newcommand{\F}{\mathbb{F}}

\newcommand{\cB}{\mathcal{B}}
\newcommand{\cZ}{\mathcal{Z}}
\newcommand{\cF}{\mathcal{F}}
\newcommand{\cC}{\mathcal{C}}

\newcommand{\cW}{\mathcal{W}}
\newcommand{\cL}{\mathcal{L}}
\newcommand{\cP}{\mathcal{P}}

\newcommand{\cH}{\mathcal{H}}

\newcommand{\te}{{\theta}}

\newcommand{\Sig}{{\Sigma}}

\newcommand{\1}{\mathbf{1}}
\newcommand{\with}{\;:\;} 
\newcommand{\supp}{\mathrm{supp}} 
\newcommand{\id}{id}

\newcommand{\myD}{D}
\newcommand{\myLD}{LD}

\begin{document} 

\begin{flushleft}
\noindent\rule{\textwidth}{0.3pt}
\textbf{\textsf{\LARGE On conformal measures and harmonic functions\\[2mm] for group extensions}}\\[6mm]
\textbf{Manuel Stadlbauer}\\[2mm]
{\footnotesize  Departamento de Matemática, Universidade Federal do Rio de Janeiro\\[-1mm] Cidade Universitária - Ilha do Fundão. 21941-909 Rio de Janeiro (RJ), Brazil.} 
\noindent\rule{\textwidth}{0.3pt}
\end{flushleft}

\bigskip

\begin{abstract}\noindent We prove a Perron-Frobenius-Ruelle theorem for 
group extensions of topological Markov chains based on a construction of $\sigma$-finite conformal measures
and give applications to the construction of harmonic functions.\\[.1cm]
\end{abstract}

\section{Introduction}

The Perron-Frobenius-Ruelle theorem is a statement about the maximal eigenvalue of an operator $L$ who preserves the cone of positive functions. Namely, it provides existence of a function $f$ in this cone and  $\nu$ in its dual, such that $Lf=\rho f$ and $L^\ast \nu =\rho \nu$, with $\rho$ referring to the spectral radius of $L$. The first result of this type was obtained by Perron in \cite{Perron:1907} as a byproduct of his analysis of periodic continued fractions. He proved that, for a strictly positive  $n \times n$-matrix $A$, the maximal eigenvalue $\rho$ is simple. Moreover, his proof reveals that there exist strictly positive vectors $x,y \in \R^n$ such that $x^t A = \rho x^t$, $Ay= \rho y$ and that $\rho^{-n} A^n$ converges to $y \cdot x^t$.
Even though there were many important contributions to the theory of positive operators in the following decades, e.g. by Doeblin-Fortet (\cite{DoeblinFortet:1937}) or Birkhoff (\cite{Birkhoff:1957}), whose methods are today standard tools in proving exponentially fast convergence of the iterates (see, e.g., \cite{AaronsonDenkerUrbanski:1993,Liverani:1995}), it was only at the end of the 60's when Ruelle obtained an analog of Perron's theorem for the one-dimensional Ising model with long range interactions from mathematical physics (\cite[Theorem 3]{Ruelle:1968}).
In the context of dynamical systems, as observed by Bowen, the result of Ruelle has the following formulation in terms of shift spaces. For a fixed  $k \in \N$, let 
\begin{align*}
\Sigma &:= \left\{ (x_i : i \in \N) : x_{i} \in \{ 1, \ldots, k\} \; \forall i \in \N \right\},\\
\theta &: \Sigma  \to \Sigma  \; (x_1,x_2, \ldots) \mapsto   (x_2,x_3, \ldots)
\end{align*}
and suppose that $\varphi: \Sigma \to (0,\infty)$ is a $\log$-Hölder continuous function with respect to the shift metric (for details, see below). The associated operator, the Ruelle operator, is then defined by 
\[ L_\varphi(f)(x) := \sum_{\te(y)=x} \varphi(y) f(y).  \]
Ruelle's theorem states that there exist a strictly positive Hölder continuous function $h$ and a probability measure $\mu$ such that 
 $L_\varphi(f)= \rho f$,  $L_\varphi^\ast(\mu)= \rho \mu$ and $\lim \rho^{-n} L_\varphi^n f = \int f d\mu \cdot h$. Furthermore, $\rho^{-n} L_\varphi^n f \to \int f d\mu \cdot h$ converges exponentially fast, which implies, among many other things, that $h$ is unique and that the measure $\mu$ is exact (and, in particular, ergodic). 
 
The aim of this note is to establish an analogue of the result for dynamical systems of the form 
 \[ T: \Sigma \times G \to \Sigma \times G, (x,g) \mapsto (\te x, g \psi(x)),\]
where $G$ is a discrete group $G$, $\Sigma$ a shift space with the b.i.p.-property as defined below and $\psi: \Sigma \to G$ a locally constant function. This kind of dynamical system is called group extension, or, as they first were considered by Rokhlin in \cite{Rohlin:1952},  Rokhlin transformation. Even though one might be tempted to think of $\psi$ as a cocycle, the probably most fruitful approach is to consider $T$ as a kind of random walk on $G$. That is, by fixing a potential function $\varphi$ which only depends on the first coordinate, $\varphi(x)$ stands for the transition probability to go from $(x,g)$ to 
$(\theta x, g \psi(x))$, which is also reflected by the fact that the Ruelle operator $\mathcal{L}_\varphi$ associated to $T$ has many similarities to the Markov operator of a random walk.
 
In this setting, it is possible to obtain the following operator theorem which is the main result (Theorem \ref{theo:family_of_measures-2}) of this note. Under a technical condition (which is satisfied if, e.g., $\Sigma$ is compact), it is shown that there exists a Lipschitz continuous map $\mu \to \nu_\mu$ from the space of probability measures  on $\Sigma \times G$ to the space of  $\sigma$-finite, conformal measures, that is $\nu_\mu$ is $\sigma$-finite and $\mathcal{L}_\varphi^\ast(\nu_\mu) = \rho \nu_\mu$. This map is constructed using the method by Denker-Urbanski in \cite{DenkerUrbanski:1991b}. By adapting ideas of Patterson (\cite{Patterson:1976}) and Sullivan  (\cite{Sullivan:1979}) from hyperbolic geometry, one  obtains by variation of $\mu$ a family 
of strictly positive, $\rho$-harmonic functions. In here, we refer to $h: \Sig \times G \to \R$ as \emph{harmonic} or \emph{$\rho$-harmonic} if  $\mathcal{L}_\varphi(h) = \rho h$. In particular, theorem \ref{theo:family_of_measures-2} gives rise to families of  $\sigma$-finite, conformal measures  $\{\nu_\mu\}$, $\rho$-harmonic, positive functions $\{h\}$, and $T$-invariant measures $\{ hd\nu_\mu\}$. Furthermore, as the conformal measures are pairwise equivalent, the function $\mathbb{K}(\mu,z) := (d\nu_\mu/d\nu)(z)$ for a fixed conformal reference measure $\nu$ is defined and, as shown in theorem \ref{theo:family_of_measures-2}, its logarithm is locally Lipschitz continuous with respect to the first coordinate and $T$-invariant with respect to the second.

It is important to point out that these families might not be one-dimensional. For example, the classical and general result of Zimmer in \cite{Zimmer:1978a} (see \cite{Jaerisch:2013a} for a version for group extensions) states that 
 ergodicity implies amenability. Hence, as $G$ not necessarily is amenable, $\nu_\mu$ might not be ergodic and the standard argument for uniqueness of conformal measures no longer is applicable. However, for the setting in here, a sharp criterium of classical flavour holds (Proposition \ref{for:ergodic_iff_divergence_type}). That is, $\nu_\mu$ is conservative and ergodic if and only if  
\[  \sum_{n=1}^{\infty} \rho^{-n}  \sum_{T^n(x,\id) = (x,\id)} \prod_{k=0}^{n-1}\varphi(\theta^k(x))
= \infty. \]
Hence, it is of interest to analyse the families of conformal measures and harmonic functions if $T$ is non-ergodic. 
In order to have explicit examples at hand, we use the fact that a random walk with independent increments can be identified with a group extension. For $T$ associated with the random walk on $\Z^d$ or the free group $\mathbb{F}_d$, 
the $d$-dimensional central limit theorem and the local limit theorem by Gerl and Woess in \cite{GerlWoess:1986}, respectively, 
allow to explicitly determine $\mathbb{K}$. For these specific examples, it turns out that the family of conformal measures is one-dimensional for $\Z^d$ and non-trivial for $\mathbb{F}_d$. 

For a further analysis of the general setting, these conformal measures are employed to construct a positive map from the space of functions $\mathcal{C}$ whose logarithm is uniformly continuous to the space of harmonic functions $\mathcal{H}$ satisfying a certain local Lipschitz condition (Theorem \ref{prop:construction of eigenfunctions}). Then, in order to at least roughly determine the behaviour of harmonic functions and $\nu$ at infinity, further ideas from probability and ergodic theory are employed. Namely, 
for a given pair $(h,\nu)$ of a positive harmonic function and a conformal measure, $hd\nu$ is invariant and therefore, the natural extension of $(T,hd\nu)$ is well-defined. Therefore, through Martingale convergence, it is possible to show (Corollary \ref{cor:decay_of_the_measure_nu}) for $G$ non-amenable and under a symmetry condition  that 
\[ \nu_\mu\left( \left\{ (x,\psi(\omega) \cdots \psi(\theta^{n-1}(\omega))) : x\in \Sig \right\} \right) = o(\rho^n), \]  
for a.e. $\omega \in \Sig$ with respect to the equilibrium measure of $(\Sig,\theta,\varphi)$. These  results also have a canonical application to the dimension theory of graph directed Markov systems, which is outlined in theorem 
\ref{theo:dimension_of_the_measure}.

\section{Topological Markov chains}

We begin with defining the basic object of our analysis, that is topological Markov chains and their group extensions. For a countable alphabet $\cW$ and a matrix $(a_{ij}:\; i,j \in \cW)$ with $a_{ij}\in \{0,1\}$ for all $i,j \in \cW$ and no rows and columns equal to $0$, let the pair $(\Sig, \te)$ denote the associated one-sided topological Markov chain given by 
\begin{align*}
\Sig& := \left\{(w_k:\;k=1,2,\ldots)\with w_k \in \cW, a_{w_k w_{k+1}}= 1 \;\forall i = 0,1, \ldots  \right\},\\
\te &: \Sig \to \Sig, \te:(w_k:\;k=1,2,\ldots) \mapsto (w_k:\;k=2,3,\ldots).
\end{align*} 
A finite sequence $w=(w_1w_2\ldots w_n)$ with $n \in \N$, $w_k \in \cW$ for $k=1,2, \ldots, n$ and $a_{w_k w_{k+1}}= 1$ for $k=1,2, \ldots, n-1$ is referred to as \emph{admissible} or as \emph{word of length $n$}, the set of  words of length $n$ will be denoted by $\cW^n$ and 
the set 
\[[w]:= \{(v_k)\in \Sig \with w_k=v_k\; \forall k=1,2,\ldots, n\}\]
is referred to as a \emph{cylinder of length $n$}. Furthermore, $|w|$ denotes the length of a word and $\cW^\infty = \bigcup_{n=1}^\infty \cW^n$ the set of all admissible words. Since $\te^n: [w]\to \te^n([w])$ is a homeomorphism, observe that the inverse $\tau_w: \te^n([w]) \to [w]$ is well defined. 

As it is well known, $\Sig$ is a Polish space with respect to the topology generated by cylinders and $\Sig$ is compact with respect to this topology if and only if $\cW$ is a finite set. Moreover, the topology generated by cylinders is compatible with the metric defined by, for $r\in (0,1)$ and $(w_k),(v_k)\in \Sig$,
\[ d_r((w_k),(v_k)) := r^{\min(i: w_i \neq v_i) - 1 }.\]
Observe that with respect to this definition, the $r^n$-neighbourhood of $(w_k) \in \Sig$ is given by the cylinder $[w_1w_2\ldots w_n]$ of length $n$. Also recall that $\Sig$  is  \emph{topologically transitive} if for all $a,b \in \cW$, there exists $n \in \N$ such that  $\te^{n}([a])\cap [b]  \neq \emptyset$ and is \emph{topologically mixing} if for all $a,b \in \cW$, there exists $N\in \N$ such that $\te^{n}([a])\cap [b]  \neq \emptyset$ for all $n \geq N$. Moreover, a topological Markov chain is said to have \emph{big images} and \emph{big preimages}
if there exists a finite set $\mathcal{I}_{\textrm{\tiny bip}}  \subset \cW$ 
such that for all $v \in \cW$, there exists $\beta_1, \beta_2 \in \mathcal{I}_{\textrm{\tiny bip}} $ such that $(v \beta_1) \in \cW^2$ and 
$(\beta_2 v) \in \cW^2$.
Finally, we say that a topological Markov chain satisfies the \emph{big images and preimages (b.i.p.) property}  if the chain is topologically mixing
and has big images and preimages (see \cite{Sarig:2003a}). Note that the b.i.p. property coincides with the notion of finite irreducibility for topological mixing topological Markov chains as introduced by Mauldin and Urbanski (\cite{MauldinUrbanski:2001}).

\paragraph{Potentials.} A further basic object for our analysis is a fixed, strictly positive function $\varphi: \Sigma \to \R$ which is referred to as a \emph{potential}. This function might be seen as weight on the preimages of a point and in many applications, $\varphi$ is defined as the conformal derivative of an underlying iterated function system. 
For $n \in \N$ and $w \in \cW^n$, set $\Phi_n := \prod_{k=0}^{n-1} \varphi \circ \te^k$ and $\Phi_w := \Phi_n \circ \tau_w$. 
We refer to $\varphi$ as  a  potential of \emph{(locally) bounded variation} if 
\[ \sup\left\{ \frac{\Phi_n(x)}{\Phi_n(y)} :\; n \in \N,\; w \in \cW^n, \; x, y \in [w]   \right\}  < \infty. \]
From now on, for positive sequences $(a_n),(b_n)$ we  will write $a_n \ll b_n$ if there exists $C>0$ with $a_n \leq C b_n$ for all $n \in \N$, and $a_n \asymp b_n$ if $a_n \ll b_n \ll a_n$. For example, the above could be rewritten by ${\Phi_{|w|}(x)} \asymp {\Phi_{|w|}(y)}$ for all $w \in \cW^\infty$ and $x, y \in [w]$.
A further, stronger assumption on the variation is related to local Hölder continuity. Recall that the $n$-th variation of a function $f: \Sig \to \R$ is defined by
\[V_n(f)=\sup\{ |f(x)-f(y)|: x_i=y_i,\, i=0,1,2,\ldots,n-1 \}.\]
The function $f$ is referred to as a \emph{locally H\" older continuous function} if there exists
$0<r<1$ and $C\geq 1$ such that $V_n(f)\ll r^n$ for all $n\geq 1$. Moreover, we refer to a locally H\" older continuous function with $\|f\|_\infty < \infty$ as a \emph{H\" older continuous function}. We now recall a well-known estimate.
For $n \leq m$, $x,y \in [w]$ for some $w \in \cW^m$, and a locally Hölder continous function $f$,  
\begin{align} \label{eq:Hoelder_estimate}
\big|\sum_{k=0}^{n-1} f\circ \te^k(x) - f\circ \te^k(y) \big| 
\ll  \frac{1}{1-r} r^{m-n}. 
\end{align}
In particular, if $\log \varphi = f$ is locally Hölder continuous, then $\varphi$ is a potential of bounded variation. Moreover, as $r^{m-n} = d(\theta^n(x),\theta^n(y))$, there exists 
$C_\varphi \geq 1$ such that  
\[ |\Phi_w(x)/\Phi_w(y) - 1| \leq C_\varphi d(x,y) \quad \hbox{ and }  \quad \Phi_w(x)/\Phi_w(y)  \leq C_\varphi \]
for all $w \in \cW^\infty$ and $x,y \in [w]$.

\paragraph{Conformal measures.} In here, due to the fact that the constructions canonically lead to $\sigma$-finite measures, we will make use of a slightly more general definition of conformality by allowing infinite measures. We refer to a $\sigma$-finite Borel measure $\mu$ as a \emph{$\varphi$-conformal} measure if 
\[ \mu(\te(A)) = \int_A \frac{1}{\varphi}d\mu  \] 
 for all Borel sets $A$ such that $\te|_A$ is injective. For $w=(w_1\ldots w_{n}) \in \cW^{n}$ and a potential of bounded variation, it then immediately follows that  
\begin{equation}\label{eq:conformal_estimate} {\mu([w])}  \asymp {\Phi_n(x)}  \; \mu(\te([w_{n}]))   \end{equation} 
for all $x \in [w]$. Note that this estimate implies that $P_G(\te,\varphi)=0$ is a necessary condition for the existence of a conformal measure  with respect to a potential of bounded variation. Moreover, if $\mu(\te([w])) \asymp 1$ (e.g., if $\mu$ is finite and $\te$ has the big image property), we obtain that  
\begin{equation}\label{eq:Gibbs}   {\mu([w])} \asymp {\Phi_n (x)} \end{equation}
for all $n \in \N$, $w \in \cW^n$ and $x \in [w]$. Also note that a probability measure satisfying (\ref{eq:Gibbs})
is referred to as a \emph{$\varphi$-Gibbs measure}.

\paragraph{Ruelle's operator, b.i.p. and Gibbs-Markov maps.} Ruelle's operator is defined, for $f : \Sigma \to \R$ in a suitable function space to be specified later, by
\[ L_\varphi (f) = \sum_{v \in \cW}   \1_{\te([v])} \cdot \varphi\circ\tau_v  \; \cdot \;f\circ\tau_v. \]
Furthermore, there is an associated action on the space of $\sigma$-finite Borel measures defined through $\int f dL_\varphi^\ast(\nu) := \int L_\varphi(f) d \nu$, for each continuous $f: \Sig \to [0,\infty)$. We then have that $\nu$ is a $\varphi/\rho$ conformal measure if and only if  $L_\varphi^\ast (\nu) = \rho \nu$. 
If, in addition, there is a measurable function $h: \Sig \to [0,\infty)$ with $L_\varphi (h) = \rho h$, then $d\mu := h d\nu$ defines an invariant, $\sigma$-finite measure, that is $\mu = \mu \circ \te^{-1}$. Moreover, for $\varphi' := \varphi h/(\rho h \circ \te)$, we have $L_{\varphi'} (1)=1$. 
 
An important consequence of the b.i.p. property is a Perron-Frobenius-Ruelle theorem in case of an infinite alphabet $\cW^1$ (see \cite{MauldinUrbanski:2001,Sarig:2003a}). That is, if  $(\Sigma,\te)$ has the b.i.p. property, $\log \varphi$ is Hölder continuous and $\|L_\varphi(1)\|_\infty < \infty$, then there exists a Gibbs measure $\mu$ and a Hölder continuous, strictly positive eigenfunction $h$ of $L_\varphi$, which is uniformly bounded from above and below. 
Moreover, in this situation, $(\Sigma,\te,\mu)$ has the \emph{Gibbs-Markov property}, that is $\mu$ is a Borel probability measure, for all $w \in \cW^1$, $\mu$ and $\mu \circ \tau_w$ are equivalent, 
$\inf \left\{\mu(\te([w]))\with w \in \cW^1\right\} >0$ and 
there exists $0<r<1$ such that, 
for all $m,n \in \N$, $v \in \cW^m$, $w \in \cW^n$ with $(vw)\in \cW^{m+n}$,
\begin{equation}  \label{eq:GM} \sup_{x,y \in [w]} \left|  \log \frac{d\mu \circ \tau_v}{d\mu}(x) -  \log\frac{d\mu \circ \tau_v}{d\mu}(y) \right| \ll r^n. 
\end{equation}
As it is well known, the action on the space of bounded continuous functions of the transfer operator with respect to $\mu$ coincides with $L_{\varphi/\rho}$ and, with $h$ referring to the function given by the Perron-Frobenius-Ruelle, $h d\mu$ is an invariant probability measure with exponential decay of correlations and associated transfer operator given by  $L_{(\varphi h)/(\rho h \circ \te)}$
(see  \cite{AaronsonDenkerUrbanski:1993,Sarig:2003a}).  

Furthermore, several arguments in here are based on an inequality in the flavour of Doeblin-Fortet or Lasota-Yorke for arbitrary topological Markov chains $(\Sigma,\theta)$ and potentials $\varphi$ such that 
is  $\log \varphi$ is locally Hölder continuous. For $f : \Sigma \to \R$, define
\[ D(f): \Sigma \to [0, \infty), \; (x_1,x_2\ldots) \mapsto \sup_{y,\tilde{y} \in [x_1]}   \frac{\left|f(y)-f(\tilde{y})\right|}{d_r(y,\tilde{y})} . \]
That is, $D(f)(x)$ is the local Hölder coefficient of the function $f$ restricted to $[a]$, with $x \in [a]$.  
Now assume that $L^n_\varphi(f)$ is well-defined. Then, for $x,y$ in the same cylinder, 
\begin{align} 
& \nonumber \left| L^n_\varphi(f)(x) - L^n_\varphi(f)(y) \right| 
\\
& =  
\nonumber
\left| \sum_{v\in \cW^n} \left(1 - \frac{\Phi_{v}(y)}{\Phi_{v}(x)} \right) \Phi_{v}(x) f\circ{\tau_v}(x) 
 + \Phi_{v}(y) \left( f\circ{\tau_v}(x)  -  f\circ{\tau_v}(y) \right)  \right| 
\\ 
\nonumber &\leq C_\varphi L_\varphi^n(|f|)(x) d_r(x,y) + r^n L_\varphi^n(D(f))(y) d_r(x,y) \\ 
\label{eq:Doeblin-Fortet}
&\leq  
C_\varphi \, d_r(x,y) \,   L_\varphi^n\left(|f| + r^n D(f)\right)(x)   
\end{align}
If, in addition, for all $a \in \cW^1$, either $f(x)=0$ for all $x \in [a]$ or $|f(x)/f(y) -1| \leq C_a d_r(x,y)$ for all $x,y \in [a]$, 
set $LD(f)(x) := 0$ in the first case and $LD(f)(x) := \sup\{ |f(x)/f(y) -1|/ d_r(x,y) : x,y \in [a] \} $ in the second case. 
By the same arguments, 
\begin{align} 
& \nonumber \left| L^n_\varphi(f)(x) - L^n_\varphi(f)(y) \right| 
\label{eq:Doeblin-Fortet-log-Hoelder} \\
\nonumber
&\leq  
C_\varphi \, d_r(x,y) \,    L^n_\varphi\left(|f| \right)(x)  
+ C_\varphi  \sum_{v\in \cW^n} \Phi_{v}(x) \left| f\circ{\tau_v}(x)\right| \left|\frac{  f\circ{\tau_v}(y)} {  f\circ{\tau_v}(x)} -1 \right| \\
&\leq  C_\varphi \, d_r(x,y) \,  L^n_\varphi\left(|f|\left( 1 + r^n LD(f) \right)   \right)(x)   
\end{align}

\section{Group extensions of topological Markov chains.}
Fix a countable  group $G$ and a map $\psi: \Sig \to G$ such that $\psi$ is constant on $[w]$ for all $w \in \cW^1$. Then, for $X:= \Sig \times G$ equipped with the product topology of $\Sig$ and the discrete topology on $G$, 
the \emph{group extension or $G$-extension} $(X,T)$ of $(\Sig, \te)$ is defined by
\[ T: X\to X, (x,g) \mapsto (\te x, g \psi(x)).\]
Note that $(X,T)$ is a topological Markov chain with respect to the alphabet $\cW^1 \times G$ and the following transition rule:  
$((a,g),(b,h))$ is admissible if and only if $(ab) \in \cW^2$ and $g \psi(a) = h$, where $\psi(a) := \psi(x)$, for some $x \in [a]$. Furthermore, set $X_g:= \Sig \times \{g\}$ and  
 \[\psi_{n}(x) := \psi(x)\psi(\te x) \cdots \psi(\te^{n-1}x)\]
for $n \in \N$ and $x \in \Sig$. Observe that $\psi_{n}:\Sig \to G$ is constant on cylinders of length $n$ which implies that 
$\psi_k(w) := \psi_{k}(x)$, for some $x \in [w]$, $k \leq n$ and $w \in \cW^n$, is well defined. If $k=n$, we will write $\psi_w := \psi_{n}(w)$.
It is then easy to see that the finite words of $(X,T)$ can be identified with  $\cW^\infty \times G$ by  
\[ ((w_0,\ldots, w_{n}),g) \equiv ((w_0,g),(w_1,g\psi_1(w)),\ldots, (w_{n},g\psi_n(w))).  \]
Also observe that topologically transitivity of $(X,T)$  implies that $\{\psi(a)\with a \in \cW^1\}$ is a generating set for $G$ as a semigroup. 

Throughout, we now fix a topological mixing topological Markov chain $(\Sig, \te)$, and a topological transitive $G$-extension $(X,T)$. Furthermore, we fix a (positive) potential $\varphi:\Sig \to \R$ with $P_G(\te,\varphi)=0$. Note that $\varphi$ lifts to a potential  $\varphi^\ast$ on $X$ by setting $\varphi^\ast(x,g):= \varphi(x)$. For ease of notation, we will not distinguish between $\varphi^\ast$ and $\varphi$. Moreover, for $v \in \cW^\infty$, the inverse branch given by $[v,\cdot]$ will be as well denoted by $\tau_v$, that is $\tau_v(x,g) := (\tau_v(x),g\psi(v)^{-1})$.
In order to distinguish between the Ruelle operator of $\te$ and $T$, these objects for the group extension will be written in calligraphic letters. That is, for $a \in \cW$, $\xi \in [a]\times \{\id\}$, $(\eta,g)\in X$, and $n \in \N$,  
\begin{align*}
\cL(f)(\xi,g)  := \sum_{v \in \cW} \varphi(\tau_v(\xi)) f\circ\tau_v(\xi,g).
\end{align*}
%
%
%
\begin{remark}\label{rem:recurrent} In the context of topological transitivity, it is natural to ask whether $(X,T)$ is ergodic with respect to the product of the Gibbs measure on $\Sig$ and the counting measure. For example,  a classical result of Zimmer in \cite{Zimmer:1978a} (see also \cite{Jaerisch:2013a}) states that ergodicity of $(X,T)$ implies that $G$ is amenable, that is, there exists a sequence $(K_n)$ of finite subsets of $G$ with $\bigcup_n K_n = G$ such that 
  \[ \lim_{n \to \infty} |g K_n \triangle K_n|/ |K_n|=0 \quad \forall g \in G,\]
 where  $ \triangle$ refers to the symmetric difference and $|\cdot|$ to the cardinality of a set. Moreover, it was shown in \cite{Stadlbauer:2013} for this class of extensions that $P_G(T) = P_G(\te)$ implies that $G$ is amenable. Hence, if $G$ is a non-amenable group, then $P_G(T) < P_G(\te)$ and $(X,T)$ is not ergodic. In particular, by bounded distortion, $T$ has to be totally dissipative. For a further criterion for ergodicity, we also refer to corollary \ref{for:ergodic_iff_divergence_type} below. Also note that the classical result of Varopoulos on recurrent groups motivates the conjecture that a group extension only can be ergodic if $G$ is a finite extension of the trivial group, $\Z$ or $\Z^2$.  
\end{remark}

\paragraph{Symmetric extensions.} In several interesting applications, group extensions are satisfying a certain notion of symmetry. In here, we will use a pathwise notion (as in \cite{Stadlbauer:2013}) in contrast to the more general notion in \cite{Jaerisch:2014a}.  Namely, we say that $(\Sig,\te,\psi)$ is \emph{symmetric} if there exists  $\cW^1 \to \cW^1$, $w \mapsto w^\dagger$ with the following properties.
\begin{enumerate}
  \item For $w \in \cW^1$, $(w^\dagger)^\dagger=w$.
  \item  \label{def:2} For $v,w \in \cW^1$, the word $(vw)$ is admissible if and only if $(w^\dagger v^\dagger)$ is admissible.
  \item   $\psi(v^\dagger) = \psi(v)^{-1}$ for all $v \in \cW^1$.
\end{enumerate} 
Moreover, we refer to  $(\Sig,\te,\psi,\varphi)$ as a \emph{symmetric group extension} if $(\Sig,\te,\psi)$ is {symmetric} and, with  
$\dagger: \cW^\infty \to \cW^\infty$  defined by $(w_1\ldots w_n)^\dagger := (w_n^\dagger\ldots w_1^\dagger)$,
    \[ \sup_{n\in \N} \sup_{x \in [w], y \in [w^\dagger]} \frac{\Phi_n(x)}{\Phi_n(y)} < \infty.\] 


\section{Conformal $\sigma$-finite measures}
As a first step towards a Ruelle theorem for group extensions, we now adapt 
ideas from \cite{Patterson:1976,DenkerUrbanski:1991b} in order to obtain invariant measures for the dual of the Ruelle operator. In contrast to  \cite{Patterson:1976,DenkerUrbanski:1991b}, the method in here gives rise to conformal $\sigma$-finite measures, which seems to be advantageous as group extensions in many cases are totally dissipative dynamical systems and therefore might not admit finite invariant measures. 
We now fix $\xi\in \Sig$ and, for  $n\in \N$, set 
\[  \cZ^n(\xi)= \sum_{\theta^n(x)=\xi, \psi_{n}(x)= \id} \Phi_n(x) = \cL^n_\varphi(\1_{X_{\id}})(\xi,\id).\]
Since the construction relies on the divergence of a power series at its radius of convergence, recall that, for a sequence of positive real numbers $(a_n)$, the radius of convergence of $\sum_n a_n x^n$ is equal to $1/\rho$ where, by Hadamard's formula, 
\[ \rho := \limsup_{n\to \infty} \sqrt[n]{a_n}.\]
We now ensure divergence at the radius of convergence by pointwise multiplication 
with a slowly diverging sequence as given by the following result. For the proof, we refer to \cite{DenkerUrbanski:1991b}.

\begin{lemma} \label{lem:analytic_fact} For a positive sequence $(a_n)$ with $\rho<\infty$,
there exists a nondecreasing sequence $(b_n :  n \in  \N)$ with $b_n \geq 1$ for all $n \in \N$ such
that $\lim_{n \to \infty} {b_n}/{b_{n+1}}=1$ and for all $s \geq 0$,
\[\sum^\infty _{n=1}b_n a_n s^{-n} \; \begin{cases} = \infty &s \leq \rho \\ <\infty& s > \rho.  \end{cases}\]
Moreover, there exists a non-increasing sequence $(\lambda(n): n \in \N)$ with $\lambda(n)\geq 1$ and $\lambda(n) \to 1$ such that 
$ b_n = \prod_{k=1}^n \lambda(k)$.  
\end{lemma}
Now suppose that $\rho = \limsup \sqrt[n]{\cZ^n(\xi)} < \infty$. Then, for $(b_n)$ given by Lemma \ref{lem:analytic_fact} applied to $a_n = \cZ^n(\xi)$, we have that   
\begin{align*}
 \cP(s) :=\sum_{n \in \N} s^{-n} b_n \cZ^n(\xi).
 \end{align*}
diverges as $s \searrow \rho$. Furthermore, for $\rho < s < \infty$, set
\begin{align}\label{eq:normalization} m_s := \frac{1}{\cP(s)} \sum_{n \in \N} s^{-n} b_n \sum_{T^n(z) = (\xi,\id)} 
 \Phi_n(x) \delta_z,
 \end{align}
where  $\delta_z$ refers to the Dirac measure supported in $z$. Note that, by construction, $m_s(X_{\id})=1$ for all $s >\rho$. In order to construct a $\sigma$-finite, conformal measure, we consider an accumulation point $\nu$ of $\{m_s\}$ in the weak$^\ast$ topology, i.e. convergence of $\int f dm_{s}$ to $\int f d\nu$ for every bounded and continuous function $f$. For ease of notation, we now identify $\Sig$ with $X_{\id}$ and,  for $B \subset \Sig$  with $T^k|_{B \times \{\id\}}$  invertible and $T^k({B \times \{\id\}}) \subset X_{\id}$, the restriction $T^k|_{B \times \{\id\}}$  with $\te^k|_B$.

\begin{lemma} \label{lemma:conf}
Assume that, for $s_l \searrow \rho$, there exists a probability measure $m$ on $\Sig$ which is the weak$^\ast$-limit of  $(m_{s_l}:l \in \N)$. Then, for each pair $(B,k)$ with $B \in \cB(\Sig)$, $k \in \N$ such that  $T^k|_{B \times \{\id\}}$ is invertible and
$T^k(B \times \{\id\})\subset X_{\id}$,  
\begin{align} \label{eq:conf1} m(\te^k(B)) = \int_B \rho^k/\Phi_k dm. \end{align}   
\end{lemma}

\begin{proof} Suppose that $B$ is a cylinder, that is $B=[w]$ for some $w \in \cW^m$ and $m>k$.    
Since $T^k$ is injective on $B \times \{\id\}$, we have, for $s > \rho$, that
\begin{align*} 
  m_{s}(\te^k(B))  &= \frac 1{\cP(s)} \sum_{n\in \N} \sum_{x\in \te^n(B) \cap E_n} \frac{b_n {\Phi_n(x)}}{s^{n}}
  = \frac 1{\cP(s)} \sum_{n\in \N} \sum_{x\in B\cap \te^{-k}(E_n)} \frac{b_n {\Phi_n(\te^k x)}}{s^{n}}\\
 &= \frac 1{\cP(s)} \sum_{n\in \N} \sum_{x\in B\cap E_{n+k}} \frac{b_{n+k}\Phi_{n+k}(x)}{s^{n+k}} \frac{b_ns^{k}}{b_{n+k} \Phi_{k}(x)}
 \end{align*}
In particular this gives
\begin{align*}
\left|  m_{s}(\te^k(B)) - \int_B \frac{s^{k}}{\Phi_{k}(x)}  d m_{s}\right| 
& \le  \frac{1}{\cP(s)} \sum_{n \in \N} \left|{\textstyle \frac{b_n}{b_{n+k}}-1} \right| \sum_{x\in B\cap E_{n+k}} \frac{b_{n+k} \Phi_{n+k}(x)}{s^{n+k}}\\
& + \frac{1}{\cP(s)}  \sum_{ n=1}^k \sum_{x\in B\cap E_{n}} b_n \Phi_{k-n}(\te^n(x))s^{k-n}.
\end{align*}  
By Lemma \ref{lem:analytic_fact}, it follows that $\lim \cP(s_l) = \infty$, and hence the second term of the right hand side tends to zero as $l \to \infty$. Since $\lim_{n \to \infty}b_n/b_{n+k} = 1$, we then obtain that the first summand also tends to zero. Moreover, by applying the Portmanteau 
theorem to the open and closed set $[w]$, it follows that (\ref{eq:conf1}) holds for $[w]$. As $\cB(\Sig)$ is generated by cylinders, the lemma follows.
\qed \end{proof}
As it seems to be impossible to show the existence of a weak$^\ast$-accumulation point of $(m_s)$ in full generality, the following condition is introduced. 
\begin{definition}
We say that the group extension $(\Sig, \te ,\varphi)$ satisfies property {(C)} if there exists $(b_n)$ as in Lemma \ref{lem:analytic_fact} and $(s_k)$ with $s_k \searrow \rho$ such that $(m_{s_k})$ converges weakly$^\ast$ to some probability measure on $X_{\id}$ as $k \to \infty$.     
\end{definition}
In order to obtain criteria for property {(C)}, recall that Prohorov's theorem states that a sequence $(m_{s_k})$ has a  weak$^\ast$-accumulation point if and only if for each $\epsilon>0$ there exists a compact set $K$ and $k_0 \in \N$ such that $m_{s_k}(K)\geq 1-\epsilon$ for all $k \geq k_0$, or in other words, if $(m_{s_k})$ is tight. In particular, if $\Sig$ is a subshift of finite type, then the property is always satisfied. 
By lifting the limit from $\Sig$ to $X$ as in \cite{DenkerYuri:2000,DenkerKiferStadlbauer:2008} we arrive at a conformal, not necessarily finite measure for $T$. 

\begin{theorem} \label{theo:conformal_general} Assume that $(\Sig, \te)$ satisfies the b.i.p. property, $\log \varphi$ is Hölder continuous, $\|L_\varphi(\1) \|_\infty <\infty$ and that $(X,T)$ be a topologically transitive group extension 
with property $(C)$. Then there exists a $\sigma$-finite, nonatomic, $(\rho/\varphi)$-conformal measure $\nu$ with $\nu(X_g)< \infty$, for each $g \in G$. 
Furthermore, there exists a sequence $(s_k)$ with $s_k \searrow \rho$, such that, for each non-negative, continuous function $f: X \to \R$,
\begin{align} \label{eq:explicit}\int f d\nu = \lim_{k \to \infty} \frac{1}{\cP(s_k)}
 \sum_{n\in \N} {b_{n}} s_k^{-n} (\cL_{\varphi}^n f)(\xi,\id). \end{align}
\end{theorem}
Before giving the proof, recall that the conditions on $(\Sig, \te ,\varphi)$ are equivalent to the existence of a probability measure $\mu$ such that $(\Sig, \te ,\mu)$ is a Gibbs-Markov map with the b.i.p. property. Hence, the above theorem holds in verbatim for  topologically transitive group extensions of Gibbs-Markov maps, with $\mu$ playing the rôle of a reference measure. 

\begin{proof} 

By property (C), there exists  $s_k \searrow \rho$ and $m$ such that $m$ is the  weak$^\ast$-limit of $(m_{s_k}:k \in \N)$.
 Using  equation (\ref{eq:conf1}) in Lemma \ref{lemma:conf}, we extend $m$ 
to a measure $\nu$ on $\cB(X)$ as follows. For $b \in \cW^1$ and $g \in G$, there exists by transitivity $j \in \N$ and $u \in \cW^{j+1}$ with $T^j([u,\id ]) = [b, g]$. The restriction of $\nu$ on $[b, g]$ is now defined by 
\[\int_{[b,g]} f(x) d\nu(x,h) := \int_{[u]} f\circ \te^j \rho^j/\Phi_j  d m, \]
for each bounded and continuous function $f: \Sig \to \R$. In particular, if $f$ is supported on $[b]$, then by the same arguments as in the proof of Lemma \ref{lemma:conf}, 
\begin{align*}
& \int_{X_g} f(x) d\nu(x,h) = \int_{[u]} f\circ \te^j \rho^j/\Phi_j  d m\\
 = & \lim_{k\to \infty} \frac{1}{\cP(s_k)}\sum_{n\in \N}{b_n}  \rho^j s_k^{-n} \sum_{x \in E_n\cap [u]} { f \circ \te^j(x)}(\Phi_j(x))^{-1} \cdot {\Phi_n(x)}\\
= & \lim_{k \to \infty} \frac{1}{\cP(s_k)}
 \sum_{n>j} {b_{n-j}} s_k^{j-n} \sum_{(y,g)\in T^{j-n}(\{(\xi,\id)\}) \cap  [b,g]} \Phi_{n-j}(y) f (y).
\end{align*}
This proves equation (\ref{eq:explicit}). Finally, using the construction of $\nu$ from $m$ and the big preimages property, it easily can be seen that $\nu(X_g)<\infty$ for each $g \in G$.
\qed\end{proof}

We now collect several immediate consequences from  conformality and the b.i.p.-property in the base.

\begin{proposition} \label{prop:measure} For the measure $\nu$ given by Theorem \ref{theo:conformal_general}, the following holds.
\begin{enumerate}
  \item If $\lim_n \cZ^n(\xi)\rho^{-n} = 0$, then $\nu(X)=\infty$.
  \item If $L_{\varphi}(1)=1$, then $ d\nu\circ T^{-1} = \rho^{-1} d\nu$.  
  \item For $w\in \cW^n$, $x \in [w]$ and $g \in G$, we have $ \rho^n \nu([w,g]) \asymp \Phi_n(x) \nu(X_{g\psi_n(x)}) $.
  \item If the extension is symmetric, then 
  \[ \nu(X_g) \asymp \nu(X_{g^{-1}}), \; \nu([w^\dagger, \psi_w^{-1}g^{-1} \psi_w ]) \asymp \nu([w,g]).\]
 \end{enumerate}
\end{proposition}

\begin{proof} The first assertion follows from (\ref{eq:explicit}) applied to $f=1$. In order to prove part 2, note that  $L_{\varphi}(1)=1$ implies that $\cL_{\varphi}(1)=1$. Hence, for $f \in L^1(\nu)$, we have 
\begin{align*}
\frac{1}{\cP(s)}
 \sum_{n\in \N} {b_{n}} s^{-n} (\cL_{\varphi}^n f\circ T)(\xi,\id)  & = 
 \frac{1}{\cP(s)}
 \sum_{n\in \N} {b_{n}} s^{-n} (\cL_{\varphi}^{n-1} f)(\xi,\id) \\
 & =  \frac{s^{-1}}{\cP(s)}  \sum_{n\in \N} \frac{b_n}{b_{n-1}}{b_{n-1}} s^{-n+1} (\cL_{\varphi}^{n-1} f)(\xi,\id). \end{align*}
Since $\cP(s) \nearrow \infty$ as $s \to \rho$ and $\lim_n {b_n/b_{n-1}}=1$ as $s \to \infty$, 
we obtain that $ \int f\circ T d\nu= \rho^{-1} \int f  d\nu$. Part 3 is a consequence of conformality and the b.i.p. property. Namely, by \eqref{eq:conformal_estimate},
\[ \rho^n \nu([w,g]) \asymp  \Phi_n(x) \nu\left(\te^n([w]) \times \{g\psi_n(x)\}\right) \leq   \Phi_n(x)  \nu\left(X_{g\psi_n(x)}\right) . \]
Furthermore, by the big images property, there exists $a \in \mathcal{I}_{\textrm{\tiny bip}}$ such that $[a] \subset \te^n([w])$. Hence, it remains to show that $\nu([a,h]) \asymp \nu(X_h)$ for all $h \in G$. By the big preimages property, for each $y \in \Sigma$, there exists $b \in \mathcal{I}_{\textrm{\tiny bip}}$ 
such that $y \in \theta([b])$. Hence, by transitivity of $T$, there exists a finite word $w$ such that $awb$ is admissible and $\psi_{awb}=\id$. Hence, $\nu([a,h]) \geq \nu([awb,h]) \asymp \nu(\te([b])\times \{h\})$ with respect to a constant only depending on $awb$, which implies that 
\[  |\mathcal{I}_{\textrm{\tiny bip}}| \nu([a,h]) \gg \sum_{b \in \mathcal{I}_{\textrm{\tiny bip}}} 
 \nu(\te([b])\times \{h\}) \geq \nu(X_h). \]    
The proof of the remaining assertion relies on a similar argument.
For each $w \in \cW^\infty$ with $\psi_w = g$ and $\xi \in \te^{|w|}([w])$, there exists by transitivity a finite word $u$ such that such that $wu$ is admissible, $\xi \in  \te^{|w| +|u|}([w^\dagger u])$ and $\psi_u = \id$. As $\mathcal{I}_{\textrm{\tiny bip}}$ is finite, $u$ can be chosen from a finite set. Hence, by the definition 
of $\nu$ and the symmetry of $\varphi$, we have $ \nu(X_{g^{-1}}) \ll \nu(X_g)$ which implies that  $ \nu(X_{g^{-1}}) \asymp \nu(X_g)$. The second assertion  follows from this and part 3.
\qed \end{proof}

\section{The Ruelle-Perron-Frobenius theorem for group extensions}

In order to prove the existence of eigenfunctions for the Ruelle operator, we make use of a well-known idea from hyperbolic geometry (see \cite{Patterson:1976,Sullivan:1979}): As the reference point for the construction  in theorem \ref{theo:conformal_general} was chosen arbitrarily, there exists a family $\{\nu_\zeta : \zeta \in X\}$ of conformal measures. It is then relatively easy to show that $\zeta \mapsto d\nu_{\zeta}/d\nu$ defines an eigenfunction, provided that $\{\nu_\zeta\}$ is a family of pairwise equivalent measures. 
In here, this approach is partially generalized by constructing a conformal measure $\nu_\mu$ for a given probability measure $\mu$ on $X$. In order to do so, recall that the Vaserstein distance $W$ of two probability measures $\mu, \tilde\mu$ is a metric compatible with the weak convergence and is equal to, by Kantorovich's duality,  
\[ W(\mu, \tilde\mu) = \sup \left\{ \int f  d(\mu-\tilde\mu): D(f)\leq 1 \right\},\]
where $D(f) := \sup\{ |f(\zeta)-f(\tilde \zeta)|/d(\zeta,\tilde \zeta) : \zeta,\tilde \zeta \in X\}$ denotes the Lipschitz coefficient with respect to the metric defined by $d((x,g),(y,g)) = d_r(x,y)$ and $d((x,g),(x,h)) = 1$ for $g \neq h$. In the following theorem, $\nu$ refers to the $\sigma$-finite, conformal measure on $X$ given by theorem \ref{theo:conformal_general} with respect to some fixed base point in $X_{\id}$.

\begin{theorem}\label{theo:family_of_measures-2} Let $(X,T)$ be a topologically transitive group extension with property (C) of a Gibbs-Markov map with the b.i.p. property. Then there exists a sequence $(s_k)$ with $s_k \searrow \rho$ such that for each  $\mu \in \mathcal{M}(X)$, 
  \begin{align} \label{eq:explicit_2}
  \nu_\mu := \lim_{k \to \infty} \frac{1}{\cP(s_k)}
 \sum_{n\in \N} {b_{n}} s_k^{-n}  (\cL_{\varphi}^n)^\ast (\mu) \end{align}
exists. Furthermore,  $\{ \nu_\mu : \mu \in \mathcal{M}(X) \}$ is a family of pairwise equivalent measures and for the Radon-Nikodym derivative $\mathbb{K}:\mathcal{P}(X) \times X \to \R$, $(\mu,z) \mapsto ({d\nu_\mu}/{d\nu})(z)$,  we have the following.
\begin{enumerate}
\item \label{enum:Hölder} There exists $D>0$ such that for $\nu$-a.e. $z \in X$, $g \in G$ and probability measures $\mu_1,\mu_2$ supported on $X_g$,
\[
|\log \mathbb{K}({\mu_1},z) - \log \mathbb{K}({\mu_2},z) |\leq D W({\mu_1},{\mu_2}).
\] 
\item For all $\mu \in \mathcal{M}(X)$,  we have $\mathbb{K}(\cL_{\varphi}^\ast(\mu),z) = \rho \mathbb{K}(\mu,z)$ for $\nu$-a.e. $z$.  
\item For each $\mu \in \mathcal{M}(X)$, the map $\mathbb{K}(\mu,\cdot)$ is $T$-invariant, that is $\mathbb{K}(\mu,z) = \mathbb{K}(\mu,T(z))$ for  $\nu$-a.e. $z \in X$. In particular, if $T$ is ergodic with respect to $\nu$, then $\nu_\mu$ is a multiple of $\nu$, $\mathbb{K}(\mu,z)$ is constant with respect to $z$ and $\{ \nu_\mu: \mu \in \mathcal{M}(X)\}$ is one-dimensional.
\end{enumerate}
\end{theorem}

\begin{remark}\label{rem:eigenfunctions} Before giving the proof, we discuss a relation to Ruelle's operator theorem. Namely, by considering the restriction $X \to  \mathcal{M}_\sigma(X)$, $x \mapsto \nu_x:=\nu_{\delta_x}$, part (ii) of the above implies that $h_z: x \mapsto \mathbb{K}(\delta_x,z)$ satisfies $\cL_{\varphi}(h_z) = \rho h_z$. Hence, the above gives rise to the construction of a family of $\sigma$-finite, conformal measures 
$\{ \nu_\mu: \mu \in \mathcal{M}(X)\}$ and a family of eigenfunctions $\{h_z : z \in X\}$. If $\nu$ is ergodic, these families are one dimensional, that is, they are subsets of $\left\{t\nu: t >0\right\}$ and $\left\{(x \to t\nu_x(X_{\id}): t > 0\right\}$, respectively.     
\end{remark}

\begin{proof} We begin with the construction of $\nu_\mu$ for the case that $\mu$ is a Dirac measure $\delta_\zeta$. So assume that  $ \zeta \in E(\xi) :=  \bigcup_{n \in \N} T^{-n}(\{(\xi,\id)\})$, for $\xi \in \Sig$ and define,  for a non-negative, continuous function $f: X \to \R$ and $s > \rho$, 
\begin{align*}
 m^s_\zeta(f)  := \frac{1}{\cP(s)}
 \sum_{n\in \N} {b_{n}} s^{-n} (\cL_{\varphi}^n f)(\zeta), 
 \end{align*}
 where ${\cP(s)}$ is given by (\ref{eq:normalization}). It follows from property (C) that $m^s_\zeta$ restricted to functions on $X_{\id}$ defines a tight family of measures, and hence that, for a suitable subsequence $(s_{k_j}: j \in \N)$ of $(s_k)$ given by property (C), 
\begin{align}
 \label{eq:other_preimage_ref_point}
  m_\zeta(f)  := \lim_{j \to \infty} \frac{1}{\cP(s_{k_j})}
 \sum_{n\in \N} {b_{n}} s_{k_j}^{-n} (\cL_{\varphi}^n f)(\zeta) 
 \end{align}
exists for each  non-negative, continuous function $f: X \to \R$. In particular, $m_\zeta$ defines a measure.
Since $E(\xi)$ is countable it is moreover possible to choose the subsequence $(s_{k_j})$ such that the limit in (\ref{eq:other_preimage_ref_point}) exists for all  $\zeta \in E(\xi)$ and $f$ non-negative and continuous. 
Moreover, as $\lim_n b_n/b_{n+k}=1$ for each $k \in \N$, it  follows that 
\begin{align}
 \label{eq:other_preimage_ref_point2} m_\zeta(f) =  \rho^{-k} \sum_{v \in \cW^k: \zeta \in \theta^k([v])\times \{g\}}  \Phi_k(\tau_v(\zeta)) m_{\tau_v(\zeta)}(f)  =   \rho^{-k} \cL_{\varphi}ˆk(m_\cdot(f)(\zeta). \end{align}
Hence, for $\zeta$ with $T^n(\zeta) = (\xi,\id)$, it follows from  (\ref{eq:other_preimage_ref_point2}) that, for each Borel set $A$,
 $m_{\xi,\id}(A) = \nu(A) \geq \rho^{-n} \Phi_n(\zeta) m_\zeta(A)$. On the other hand, it follows from transitivity that there exist $v \in \cW^m$ and $m \in \N$ such that $\tau_v(\zeta)$ and $(\xi,\id)$ are in the same cylinder. Hence, by combining the above argument with bounded distortion, we obtain that
\begin{align*} \label{eq:equivalence_of_densely_defined_measures}
\rho^{m}  \Phi_m(\xi)^{-1} m_\zeta(A) \gg \nu(A) \geq \rho^{-n}  \Phi_n(\zeta) m_\zeta(A). \end{align*} In particular, the measures are equivalent and the Radon-Nikodym derivative $\mathbb{K}(\zeta,\cdot):= dm_\zeta/d\nu$ exists and is a.s. strictly positive. 

We now prove that  $m_{\zeta_1}(A) \asymp  m_{\zeta_2}(A)$ whenever the second coordinates coincide, that is $\zeta_1, \zeta_2  \in E(\xi) \cap X_g$ for some $g \in G$. In order to do so, assume that $\zeta_2 \in [a,g]$ for some $a \in \mathcal{I}_{\textrm{\tiny bip}}$. By the b.i.p.-property, there exist $b \in \mathcal{I}_{\textrm{\tiny bip}}$ and $h \in G$ such that $\zeta_1 \in T([b,h])$ and by transitivity a finite word  $w$ such that $awb$ is admissible with $\psi_{awb}=\id$. As above, it follows that $  \Phi_{|awb|}(x)  m_{\zeta_2}(A) \ll  m_{\zeta_1}(A)$ for any $x \in [awb]$. Hence, as  
$\mathcal{I}_{\textrm{\tiny bip}}$ is finite, $m_{\zeta_2}(A) \ll  m_{\zeta_1}(A)$ with respect to a constant which does not depend on $\zeta_1$ and $a \in \mathcal{I}_{\textrm{\tiny bip}}$.

In order to prove the opposite direction, for each $b \in \mathcal{I}_{\textrm{\tiny bip}}$ choose $x_b \in [b]$. Also note that, for each $v \in \cW^1$ with $\zeta_1 \in \theta([v])\times \{g\}$, there exists $b(v) \in \mathcal{I}_{\textrm{\tiny bip}}$ such that $vb(v)$ is admissible. As 
$\varphi(\tau_v(\zeta_1)) \asymp \varphi(\tau_v((x_{b(v)},g))$, we have by the above that
\begin{align*}
m_{\zeta_1}(A) & = \rho^{-1} \sum_{v : \zeta_1 \in \theta([v])\times \{g\}}  \varphi(\tau_v(\zeta_1)) \;  m_{\tau_v(\zeta_1)}(A)  \\
& \asymp \rho^{-1} \sum_{v : \zeta_1 \in \theta([v])\times \{g\}}  \varphi(\tau_v((x_{b(v)},g))) \;m_{\tau_v((x_{b(v)},g))}(A)  \\
& \leq \sum_{b \in  \mathcal{I}_{\textrm{\tiny bip}}} m_{(x_b,g)}(A) \ll | \mathcal{I}_{\textrm{\tiny bip}}| \; m_{\zeta_2}(A).
\end{align*}
Hence,  $m_{\zeta_1}(A) \asymp  m_{\zeta_2}(A)$ which implies that
 \begin{equation}\label{eq:estimate_inside_the_pot}
 \sup\left(\left\{ \frac{\mathbb{K}((x_1,g),z)}{\mathbb{K}((x_2,g),z)} \with  (x_1,g),(x_2,g) \in E(\xi), g \in G, z \in X \right\}\right) < \infty.
 \end{equation}
In order to extend $\mathbb{K}(\cdot,\cdot)$ to a globally defined function, we now show that  $\zeta \mapsto \mathbb{K}(\zeta,z)$ is $\log$ Hölder. For $k,n \in \N$, $\zeta_1,\zeta_2 \in [a,g] \cap E(\xi)$, $a \in \cW^1$, $b \in \cW^k$ with $k \leq n$ and $h \in G$, we obtain by \eqref{eq:Doeblin-Fortet} that 
\begin{align} \label{eq:hölderestimate_for_the_eigenfunction}
 |\cL_{\varphi}^n(\1_{[b,h]})(\zeta_1) -  \cL_{\varphi}^n(\1_{[b,h]})(\zeta_2)| 
  & \leq C_\varphi d(\zeta_1,\zeta_2) \cL_{\varphi}^n(\1_{[b,h]})(\zeta_1),
\end{align} 
with $C_\varphi$ only depending on the H\"older constant of $\varphi$.
Hence, 
\[|m^s_{\zeta_1}([b,h]) - m^s_{\zeta_2}([b,h])| \leq C_\varphi m^s_{\zeta_1}([b,h]) d(\zeta_1,\zeta_2)\]
and, by taking the limit, 
\[|m_{\zeta_1}([b,h]) - m_{\zeta_2}([b,h])|  \leq C_\varphi   m_{\zeta_1}([b,h]) d(\zeta_1,\zeta_2).\]
 Since cylinder sets are generating the Borel algebra and are stable under intersections it follows by taking the limit as 
$[b,h] \to z \in X$ that $|(dm_{\zeta_1}/dm_{\zeta_2})(z) - 1| \ll d(\zeta_1,\zeta_2) $ for $\nu$-a.e. $z \in X$. 
Furthermore, as $\zeta_1,\zeta_2 \in X_g$, it follows from \eqref{eq:estimate_inside_the_pot} that $dm_{\zeta_1}/dm_{\zeta_2} \asymp 1$. Hence, $|\log  (dm_{\zeta_1}/dm_{\zeta_2})(z) | \ll d(\zeta_1,\zeta_2)$, which proves that 
the function $\zeta \mapsto \log \mathbb{K}(\zeta,z)$ is Lipschitz continuous on $E(\xi)\cap [a,g]$ with respect to a Lipschitz coefficient which is independent from $z$ and $[a,g]$.
By a further application of \eqref{eq:estimate_inside_the_pot}, there is a uniform bound for  $|\log  (dm_{\zeta_1}/dm_{\zeta_2})(z) |$ which is independent from $z$ and $g$. 
As $E(\xi)$ is dense by transitivity, there exists a unique locally Lipschitz continuous extension of $\zeta \mapsto \log \mathbb{K}(\zeta,z))$ to $X$. By taking the exponential of this extension, we obtain a globally defined function which, for ease of notation, will also be denoted by $\mathbb{K}(\cdot,\cdot)$. As the function has the same regularity as the one defined on $E(\xi)$, we have shown that
 there exists $D>0$ such that, for all $g \in G, \, \zeta_1,\zeta_2 \in X_g$ and $\nu$-a.e.  $z \in X$, 
\[ \left|\log \mathbb{K}(\zeta_1,z) - \log \mathbb{K}(\zeta_2,z)\right| \leq D d(\zeta_1,\zeta_2).
 \]

In order to obtain the representation (\ref{eq:explicit_2}), note that the construction of $m_\zeta$ through (\ref{eq:other_preimage_ref_point}) extends to all $\zeta \in X$ by the estimate (\ref{eq:hölderestimate_for_the_eigenfunction}) and the fact that $E(\xi)$ is dense in $X$. 
The next step is to verify that (\ref{eq:other_preimage_ref_point}) extends to an arbitrary Borel probability measure $\mu$ on $X$. In analogy to the above, define    
\begin{align*}
 M^s_\mu(f)  &:= \frac{1}{\cP(s)}
 \sum_{n\in \N} {b_{n}} s^{-n} \int f d (\cL_{\varphi}^n)^\ast(\mu)  \\
  &= 
 \frac{1}{\cP(s)} \sum_{n\in \N} {b_{n}} s^{-n} \int  \cL_{\varphi}^n(f)d\mu  = \int m^s_\zeta(f) d\mu(\zeta),
 \end{align*}
where the last equality follows from monotone convergence. By a further application of monotone convergence, it follows that $\lim_{k}  M^{s_k}_\mu(f) = \int m_\zeta(f) d\mu$ which proves that  \eqref{eq:other_preimage_ref_point} defines a measure and that $\nu_\mu(f):= \int m_\zeta(f) d\mu$.  Moreover, as 
\begin{align}
 \label{eq:radon-nikodym for nu_mu}
 \int f d\nu_\mu = \int m_\zeta(f) d\mu = \iint f(z) \mathbb{K}(\zeta,z) d\nu(z) d\mu(\zeta),
\end{align}
it follows that $d\nu_\mu /d\nu = \int \mathbb{K}(\zeta,z) d\mu(\zeta)$, which will also be denoted by $\mathbb{K}(\mu,z)$, by a slight abuse of notation. This finishes the proof of the existence of $\nu_\mu$. Part 1 of the theorem then follows from the definition of $W$ through Kantorovich's duality. 


In order to prove part 2, note that (\ref{eq:explicit_2}) implies that, for $\zeta \in X$ and each positive and continuous function $f$, 
that 
\begin{align} \label{eq:L-invariance_of_the_integral}
& \int f(z) \mathbb{K}(\zeta,z) d\nu(z)  = \int f d\nu_\zeta = \lim_{k \to \infty} \frac{1}{\cP(s_k)}
 \sum_{n\in \N} {b_{n}} s_k^{-n} (\cL_{\varphi}^n (f))(\zeta) \\
\nonumber
=& \sum_{v \in \cW} \varphi\circ\tau_v(\zeta)  \lim_{k \to \infty} \frac{1}{\cP(s_k)}
 \sum_{n\in \N} {b_{n}} s_k^{-n} (\cL_{\varphi}^{n-1} (f))(\tau_v(\zeta))\\
\nonumber
=& \sum_{v \in \cW} \varphi\circ\tau_v(\zeta) \rho^{-1} \int f d\nu_{\tau_v(\zeta)} 
  = \sum_{v \in \cW} \varphi\circ\tau_v(\zeta) \rho^{-1} \int f \mathbb{K}({\tau_v(\zeta)},\cdot) d\nu\\
\nonumber  
  =& \rho^{-1}  \int f(z) \cL_{\varphi}(\mathbb{K}(\cdot,z))(\zeta) d\nu(z) 
\end{align} 
where the last identity follows from monotone convergence. Hence, by \eqref{eq:radon-nikodym for nu_mu},
\begin{align*} \rho \int f \mathbb{K}(\mu,\cdot) d\nu & =  \int f(z) \int \cL_{\varphi}(\mathbb{K}(\cdot,z))(\zeta) d\mu(\zeta) d\nu(z) \\
& = \int f(z) \int \mathbb{K}(\cdot,z)d\cL_{\varphi}^\ast(\mu)  d\nu(z) = \int f \mathbb{K}(\cL_{\varphi}^\ast(\mu),\cdot)d\nu.
\end{align*}
As $f$ is arbitrary, $ \rho \mathbb{K}(\mu,z)  =  \mathbb{K}(\cL_{\varphi}^\ast(\mu),z)$ almost surely, which is part 2 
 of the theorem. 

For the proof of part 3, note that $X$ is a Besicovitch space and that each $\nu_\mu$ is conformal. Therefore, we have for $\nu$-a.e. $((w_i),g)\in X$, that 
\[ \mathbb{K}(\mu,((w_i),g)) = \lim_{n \to \infty}  \frac{\nu_\mu([(w_1\ldots w_n),g])}{\nu([(w_1\ldots w_n),g])}
 = \lim_{n \to \infty}  \frac{ \rho^{-1}\int_{[(w_2\ldots w_n),g\psi_{w_1}]} \varphi d\nu_\mu}{\rho^{-1} \int_{[(w_2\ldots w_n),g\psi_{w_1}]} \varphi d\nu}.  \] 
 It hence follows from continuity of $\varphi$ that $\mathbb{K}$ is $T$-invariant in the second coordinate. The second statement  is a standard application of the ergodic theorem. \qed 
\end{proof} 

We now give a brief characterization of the measures given by above theorem 
in case of an ergodic extension (as e.g. in example \ref{ex:polya} below for $d=1,2$). For ease of exposition, we assume that the base transformation is a Gibbs-Markov map with respect to the invariant probability $\mu$ on $\Sig$. In this situation, the product measure $\mu_G$ of $\mu$ and the counting measure on $G$ clearly is $1/\varphi$-conformal and $T$-invariant, i.e.  $\mu_G = \mu_G \circ T^{-1}$. However, note that $\mu_G$ in many cases is totally dissipative, e.g. if $G$ is non-amenable (\cite{Zimmer:1978a,Jaerisch:2013a}). 

If $T$ is conservative with respect to $\mu_G$, then $T$ also is ergodic and $\sum_n  \cZ_w^n(\xi) = \infty$ (see \cite{AaronsonDenkerUrbanski:1993,Aaronson:1997} and the proof of proposition \ref{for:ergodic_iff_divergence_type} below). In particular, $\rho=1$ and $G$ is amenable by a result in \cite{Stadlbauer:2013}. Since $\mu_G$ and $\nu_\zeta$ are both  $1/\varphi$-conformal, as observed by Sullivan, $d\nu / d\mu_G$ exists, is $T$-invariant and hence constant. This then implies that 
 the measures $\nu_\xi$ are all multiples of the product measure $\mu_G$. 
If $T$ is conservative with respect to $\nu$ and $\rho \leq 1$, then the same arguments show that the measures $\nu_\xi$ are again all multiples of $\nu$. In this situation, a  result by Jaerisch (\cite{Jaerisch:2013a}) shows that the invariant measure $h(x,z)d\nu(x)$ is unique and is the product of another measure on $\Sig$ and counting measure on $G$.

As a corollary of the existence of $\rho^{-1}\cL_\varphi$-invariant functions as shown in Remark \ref{rem:eigenfunctions}, one obtains the following  criterion of classical flavor for ergodicity.

\begin{proposition} \label{for:ergodic_iff_divergence_type} The map $T$ is either conservative or totally dissipative with respect to $\nu$. If $T$ is conservative, then $T$  is ergodic. Furthermore, $T$ is conservative and ergodic if and only if 
\[\sum_n \rho^{-n} \cZ^n(\xi) = \infty.\] 
\end{proposition}

\begin{proof}
Observe that $T$ is a transitive topological Markov chain and that 
it follows from 
\[ (d\nu/d\nu \circ T)(x,g) =  \rho\phi(x)\]
that $d\nu/d\nu \circ T$ is a potential of bounded variation. Hence, $(T,\nu)$ is a Markov fibered system with the bounded distortion property as in \cite{AaronsonDenkerUrbanski:1993}. In particular, $(T,\nu)$ either is totally dissipative or conservative and if $(T,\nu)$ is conservative, then it is ergodic. Note that $\rho^{-1}\cL_\varphi$ acts as the transfer operator on $L^1(\nu)$. It hence follows from the definition of the transfer operator that, for all $W$ measurable  and $n \in \N$,
\[  \int \1_{W} \rho^{-n} \cL_\varphi^n(x,g) d\nu(x,g) = 
\int \1_{W} \circ T^n \1_{X_{\id}} d\nu = \nu\left(T^{-n}(W) \cap X_{\id} \right). 
 \]
Now assume that $\sum_n \rho^{-n} \cZ^n(\xi) = \infty$. It follows from bounded variation and transitivity that the sum diverges for all $\xi \in \Sig$. For $W :=\{z \in X_{\id} \with T^n(z) \notin X_{\id}  \forall n \geq 1\}$, we hence have that $\nu(W) = 0$. Hence, the first return map
\[ T_{X_{\id}}: X_{\id} \to X_{\id}, \; (x,\id) \mapsto T^{n_x}(x,g), \hbox{  } n_x := \min\{n \geq 1 : T^n(x,\id)\in X_{\id} \} \]
is well defined. By substituting $\nu$ with an equivalent, invariant measure given by the above theorem, an application of Poincaré's recurrence theorem gives that $T_{X_{\id}}$ is conservative. It is then easy to see that $T$ also is conservative, and hence ergodic.
The remaining assertion is a consequence of the standard result in ergodic theory, that $T$ is ergodic and conservative if and only if $\sum_n \rho^{-n}\cL^n_\varphi(f)$ diverges for all $f \geq 0$, $\int f d\nu >0$ (see \cite[Prop. 1.3.2]{Aaronson:1997}).   \qed 
\end{proof}

\section{Harmonic functions} 

By applying theorem \ref{theo:family_of_measures-2} to Dirac measures, it is possible to construct a map $\Theta:\cC \to \cH$  from a subspace of the continuous functions to a subspace of $\rho$-harmonic functions. In here, we refer to $f:X \to \R$ as \emph{$\rho$-harmonic} if $\cL_\varphi(f)=\rho f$. In order to define $\cC$, fix a reference point $\xi_0 \in X_{\id}$ and set $\nu_\mathbf{o}:= \nu_{\xi_0}$.
The space $\cC$ is now defined by
\begin{align*}
\cC & :=  \left\{f:X \to \R \with \nu_\mathbf{o}(|f|)< \infty, \lim_{n\to \infty} C_n(f)=0  \right\}, \hbox{ where }\\
C_n(f) &:= \inf\left(\left\{C : |f(z_1)-f(z_2)| \leq C|f(z_1)| \forall z_1,z_2 \in [w,g], w \in \cW^n, g \in G \right\}\right).
\end{align*}
The  space might be alternatively characterised as the space of $\log$-uniformly continuous functions with an integrability condition. Namely, 
if $C_n(f)< \infty$, then for $[w,g], w \in \cW^n$ and $g \in G$ either $f|_{[w,g]} = 0$  or $f(z) \neq 0$ for $z \in [w,g]$. In particular, with $0/0 =1$, it follows that 
\[ \left|f(z_1)/f(z_2) -1 \right| < C_n(f), \quad \forall z_1,z_2 \in [w,g]. \]     
Hence, if $C_n(f)< 1$, then $f(z_1)/f(z_2)>0$, that is the sign of $f$ is constant on $[w,g]$. These arguments show that $f \in \cC$ if and only if $\log f_+$ and $\log f_-$ are uniformly continuous, with $f_{\pm}$ referring to the strictly positive and negative parts of $f$ and $\{f \neq  0\}$ is a union of cylinders of length $n$, for some $n$ depending on $f$.   

In order to define $\cH$, recall that $d_r$ refers to the shift metric on $X$, with $r \in (0,1)$ adapted to the Hölder continuity of $\log \varphi$. In order to be able to not only consider positive $\rho$-harmonic functions, the following coefficients for the local regularity of a function $f: X \to \R$ are useful.
\begin{align*}
\myD_x(f) & :=  \sup \left\{  |f(x)-f(z)|/d_r(x,z) \with  d_r(x,z) <1 \right\} \\
\myLD(f) & := \sup \left\{  \left|{\textstyle \frac{f(z_1)}{f(z_2)} -1 }\right|/d_r(z_1,z_2) \with  d_r(z_1,z_2) <1 \right\}
\end{align*}
The space $\cH$ is now defined through a control of the local Lipschitz constant $\myD_x(f)$ as follows.
\begin{align*}
\cH^+ & :=  \left\{ (f:X \to [0,\infty)) : \cL_\varphi(f) = \rho f,\;  \myLD(f)< \infty \right\},\\
\cH  &  :=  \left\{(f:X \to \R) : \cL_\varphi(f) = \rho f,\;  \exists h \in \cH^+ \hbox{ s.t. } \myD_x(f) \leq h(x) \forall x \in X \right\}.
\end{align*}
The map $\Theta$ is then defined by, for $f \in \cC$,
\[ \Theta(f)(z) := \nu_z(f) = \int \mathbb{K}(\delta_z,y)f(y) d\nu_\mathbf{o}(y). \]
Based on a slightly more involved version of the argument used in the proof of $\log$-Hölder continuity of $\mathbb{K}$ in theorem \ref{theo:family_of_measures-2} we are now in position to prove that $\Theta$ is well defined and that $\myLD$ is always bounded by 
\[ 
C_\varphi:= \sup \left\{ \left| {\textstyle \frac{\Phi_{n}\circ\tau_v(z_1)}{\Phi_{n}\circ \tau_v(z_2)}-1}\right|/d_r(z_1,z_2) : d_r(z_1,z_2) <1 \right\} < \infty.
\] 

\begin{theorem}\label{prop:construction of eigenfunctions} The map $\Theta: \cC \to \cH$ is well defined. If $f \in \cH$ and $f \geq 0$, then $\myLD(f) \leq C_\varphi$ and, in particular, $f \in \cH^+$.
\end{theorem}

\begin{proof} Suppose that $f \in \cC$. By applying the arguments in (\ref{eq:L-invariance_of_the_integral}) to $f$ shows that $\cL_\varphi(\Theta(f)) = \rho \Theta(f)$. Hence, it remains to obtain a bound on $D_x(f)$. For ease of notation, set $f_v := f \circ \tau_v$, for $v \in \cW^n$ and $n \in \N$. 
Suppose that $z_1,z_2 \in [w,g]$ with $w \in \cW^k$, $g \in G$ and that $n$ is sufficiently large such that for all $v \in \cW^n$, either $f_v(z_1) = f_v(z_2)=0$ or $f_v(z_1), f_v(z_2) \neq 0$. Setting $0/0 :=1$ and $A_n:= \sup_{v \in \cW^n} \left|{\textstyle  \frac{f_v(z_1)}{f_v(z_2)}-1} \right|$, we obtain by a similar argument as in \eqref{eq:Doeblin-Fortet}  that
\begin{align*}
& \left| \cL^n_\varphi(f)(z_1) - \cL^n_\varphi(f)(z_2) \right| \\
 \leq &  
  \sum_{v\in \cW^n}  \left|\left(  \Phi_{n,v}(z_1) - \Phi_{n,v}(z_2) \right) f_v(z_1)  \right| +  
  \sum_{v\in \cW^n}  \left| \Phi_{n,v}(z_2)  \left( f_v(z_1)  -  f_v(z_2) \right)  \right|\\ 
\leq & \sum_{v\in \cW^n}   \left| \left( {\textstyle 1- \frac{\Phi_{n,v}(z_2)}{\Phi_{n,v}(z_1)}} \right)  \Phi_{n,v}(z_1) f_v(z_1)  \right| +
 \sum_{v\in \cW^n}  \left|  \Phi_{n,v}(z_2) f_v(z_2) \left({\textstyle  \frac{f_v(z_1)}{f_v(z_2)}-1} \right)  \right|\\
\leq & C_\varphi d_r(z_1,z_2) \cdot   \cL^n_\varphi(|f|)(z_1)   +   A_n \cdot  \cL^n_\varphi(|f|)(z_2).
\end{align*}
Since $ \lim_{n \to \infty} A_n = 0$, we have
\begin{align*}
|\Theta(f)(z_1) - \Theta(f)(z_2)|   & = |\nu_{z_1}(f) - \nu_{z_2}(f)| \\
& \leq  C_\varphi d_r(z_1,z_2) \nu_{z_1}(|f|) = C_\varphi d_r(z_1,z_2) \Theta(|f|)(z_1).
\end{align*}
Hence, $\myD_{z_1}(f) \leq C_\varphi \Theta(|f|)(z_1)$. By dividing with $\Theta(|f|)$ and substituting $f$ with $|f|$, the same argument shows that $\myLD(\Theta(|f|))\leq C_\varphi$. In particular, $\Theta(|f|) \in \cH^+$.  
Now assume that $f \in \cH$ and $f \geq 0$. Then there exists $\hat{h} \geq 0$ with $D_z(h)\leq \hat{h}$ and $\cL_\varphi(\hat{h}) = \rho \hat{h}$. By similar arguments,
\begin{align*}
|h(z_1) - h(z_2)| &=  \rho^{-n}\left| \cL^n_\varphi(h)(z_1) - \cL^n_\varphi(h)(z_2) \right|\\
& \leq   \rho^{-n} \left(  C_\varphi d_r(z_1,z_2) \cL^n_\varphi(h)(z_1) +   r^n d_r(z_1,z_2) \cL^n_\varphi(\myD_\cdot(h))(z_2)  \right) \\
&\leq  C_\varphi d_r(z_1,z_2) \left( h(z_1) +  r^n \hat{h}(z_2)\right).
\end{align*}
Since $n$ is arbitrary and $r \in (0,1)$, $\myLD(h)<C_\varphi$. \qed
\end{proof}

The classical Martin boundary of a random walk on a group is a quotient of the space of paths, where two paths $(g_k)$, $(h_k)$ in $G$ are identified if $\lim_k K(\cdot,g_k) = \lim_k K(\cdot,h_k)$, where $K$ refers to the Martin kernel (see, e.g., \cite{Woess:2009}).
In the context of group extensions, the natural candidate for a path in $G$ is given by $(\psi_k(x))$, for some $x \in \Sig$, whereas the function $(z,g) \mapsto  \nu_z(X_g)/\nu_\mathbf{o}(X_g)$ might serve as the analogue of the Martin kernel.  

Here, the situation is different. Assume that $(x,g)=((w_k),g) \in X$. Using the conformality of $\nu$  in proposition \ref{prop:measure}, we have by theorem \ref{theo:family_of_measures-2} that, for $f_n := \1_{[w_1 \ldots w_n,g]}/\nu_\mathbf{o}([w_1 \ldots w_n,g])$, 
\begin{align} \nonumber \label{eq:convergence_to_the_kernel}
 \frac{\nu_z(X_{g\psi_n(x)})}{\nu_\mathbf{o}(X_{g\psi_n(x)})} \asymp \frac{\nu_z([w_1 \ldots w_n,g])}{\nu_\mathbf{o}([w_1 \ldots w_n,g])}
 = \Theta\left(
 f_n
 \right)(z) 
 \xrightarrow{n \to \infty} \mathbb{K}(z,(x,g)). 
\end{align}

\subsection{Natural extensions and immediate implications}
In order to obtain  information on the asymptotic behavior of elements of $\cH$, we now employ ideas from the theory of Markov processes, which are similar but somehow dual to the ones for Markov maps. Namely, in order to obtain a stochastic process associated with $(X,T)$, we consider the process with transition probability $(dm\circ\tau_v/dm)(x)$ for transitions from $x$ to $\tau_v(x)$, where $m$ is an $T$-invariant measure. Hence, the appropriate object are the left-infinite sequences with respect to an invariant measure $\widehat{m}$ constructed from $m$. That is, the stochastic process is the left half of the natural extension of $(X,T,m)$ whose construction in case of an underlying shift space we recall now.
Set
\begin{align*}
Y & := \left\{ ((w_i,g_i) : i \in \Z) : w_i \in \cW^1, g_i \in G, a_{w_i w_{i+1}}=1, g_{i+1} = g_i\psi(w_i)  \right\}, \\
S & :Y \to Y, \; ((w_i,g_i)) \to ((w_i',g_i')), \hbox{ with }  w_i' = w_{i+1},\;  g_i'= g_{i+1} \; \forall i \in \Z.
\end{align*}
In other words, $S$ is the left shift on the two sided shift space $Y$. The cylinder sets of $Y$ are given by, for $(w_0w_1\cdots w_n)\in \cW^{n+1}$, $h_i \in G$ and $k \in \Z$,
\begin{align*} &[((w_0,g_0)\cdots(w_n,g_n))]_k \\ & \quad :=  \left\{ ((v_i,h_i)) \in Y :  (v_{k+j},h_{k+j}) = (w_{j},g_{j}), \hbox{ for } j=0,1,\ldots,n    \right\}. \end{align*}
If $m$ is $T$-invariant, then $\widehat{m}([((w_0,g_0)\cdots(w_n,g_n))]_k):= m([(w_0 w_1\cdots w_n),g_0])$ defines a measure   
 $\widehat{m}$ on $Y$. As it easily can be seen, we then have, for  
\[\pi : Y \to X, ((w_i,g_i)) \to ((w_i: i \geq 0),g_0),\]
that $\pi \circ S = T \circ \pi$, $\widehat{m} = m \circ \pi^{-1}$, $S$ is invertible, $\widehat{m}$ is $S$-invariant and $(Y,S,\widehat{m})$ is  minimal in the sense that the $\sigma$-algebra $\cF$ generated by the cylinder sets of $Y$ is generated by $\left\{ S^n(\pi^{-1}(\mathcal{B})): n \in \Z \right\}$, with $\mathcal{B}$ referring to the $\sigma$-algebra generated by the cylinder sets of $X$. In particular, $(Y,S,\cF,\widehat{m})$ is the natural extension of $(X,T,\cB,m)$ (see, e.g., \cite{CornfeldFominSinaui:1982}). 

Observe that there are   several canonical choices for the invariant measure $m$. Either $\mu$ is  $\theta$-invariant and $m$ is the product $\mu_G$ of $\mu$ and the counting measure on $G$, or $dm = h d \nu_\mathbf{o}$, for some $h\in \cH^+$. However, in both cases, it is possible to identify martingales with respect to the filtration $(\cF_n: n \in \N)$, where $\cF_n := S^{n}\circ \pi^{-1}(\cB)$. We begin with the analysis of $(Y,S)$ with respect to $\widehat{\mu}_G$.

\begin{proposition} \label{prop:martingale_for_mu} 
Suppose that $\mu$ is  $\theta$-invariant and that $h \in \cH^+$. Then, for $\widehat\mu_G$-a.e. $z \in Y$, 
\[h_\infty(z) :=  \lim_{n\to \infty} \rho^{-n} h\circ\pi\circ S^{-n}(z)\]
exists. If $\rho<1$, then $h_\infty=0$, and if $\rho=1$, then  $h_\infty =  h_\infty \circ S$ and $h_\infty < \infty$ a.s.
\end{proposition}

\begin{proof} Set $W_n :=  \rho^{-n} h \circ \pi \circ S^{-n}$. Since  $\widehat\mu_G$ is $S$-invariant, $\int f\circ T g d\mu_G = \int f \cL(g) d\mu_G$ and $\cL_\varphi(h) = \rho h$, we have for all $A \in \cB$ that 
\begin{align*}
& \int_{S^n(\pi^{-1}(A))} \mathbb{E}(W_{n+1}|\cF_n) d\widehat\mu_G \\
  = &\rho^{-n-1}  \int \1_{A}\circ \pi \circ S^{-n} \;   h\circ \pi \circ S^{-n-1} d\widehat\mu_G\\  
  = &\rho^{-n-1}  \int \1_{A}\circ T \; h  d\mu_G = \rho^{-n} \int \1_A h d\mu_G \\
  = &\rho^{-n} \int \1_A\circ \pi \circ S^{-n}  h \circ \pi \circ S^{-n}  d\widehat \mu_G =
 \int_{S^n(\pi^{-1}(A))} W_n d\widehat\mu_G.
\end{align*}
Hence, $\mathbb{E}(W_{n+1}|\cF_n)=W_n$ and $(W_n,\cF_n)$ is a positive martingale. In particular, $h_\infty:= \lim_n W_n$  by Doob's convergence theorem. As it easily can be verified, we have $ h_\infty = \rho h_\infty \circ S $. Furthermore, by 
Fatou's Lemma and the martingale property, $\int_{\pi^{-1}A} h_\infty d\widehat{\mu} \leq \int_A h d\mu$ for all measurable sets $A \subset X$, which implies that $h_\infty < \infty$ a.s.
\qed
\end{proof}
By applying the proposition to $\Theta(\1_{X_{\id}})$, we obtain the decay of $\nu$ along $\mu$-a.s. path as $n \to -\infty$. If the extension is symmetric, the result also transfers to paths with $n \to \infty$.
\begin{corollary} \label{cor:decay_of_the_measure_nu} If $\rho<1$ and $\mu_G$ is invariant, then, for $\widehat{\mu}_G$-a.e. $((w_i),g)\in Y$,  
\[ 
\lim_{n \to \infty} \nu_\mathbf{o}(X_{g \psi_{w_{-n}} \cdots \psi_{w_{-1}}})/\rho^n = 0, \quad
\lim_{n \to \infty} \frac{\nu_\mathbf{o}([w_{-n}\ldots w_{-1},g])}{\mu([w_{-n}\ldots w_{-1}])}=0.
\]
Moreover, if the group extension is symmetric, then for $\mu$-a.e. $x \in \Sig$ and $g \in G$, 
\[
\lim_{n \to \infty} \nu_\mathbf{o}(X_{\psi_{n}(x)})/\rho^n = 0, \quad
\lim_{n \to \infty} \frac{\nu_\mathbf{o}([w_{1}\ldots w_{n}, \psi_{n}(x) g \psi_{n}(x)^{-1}])}{\mu([w_{1}\ldots w_{n}])}=0.
\]
\end{corollary}
\begin{proof} The first two assertions follow from  $\nu_{(x,g)}(X_{\id}) \asymp \nu_\mathbf{o}(X_{g^{-1}})$ and (iii) of proposition \ref{prop:measure}, whereas the last two assertions are a consequence of the fact that $Y \to Y$, $((w_i),g) \mapsto ((w^\dagger_{-i}),g)$ is a non-singular automorphism. \qed
\end{proof}

By considering  the natural extension of the invariant version of 
$\nu_\mathbf{o}$, we obtain a further convergence. That is, as the measure $dm_h := h d\nu_\mathbf{o}$ is $T$-invariant, there exists a unique extension to an invariant, $\sigma$-finite, $S$-invariant measure $\widehat{m}_h$ on $Y$. The analogue of proposition \ref{prop:martingale_for_mu} is as follows.
\begin{proposition} \label{prop:martingale_for_nu} 
Suppose that $f,h \in \cH$, $h>0$ such that $\|f/h\|_\infty < \infty$. Then, for   
 $\widehat{m}_h$-a.e. $z \in Y$, 
\[\Xi_h(f)(z) :=  \lim_{n\to \infty} \frac{f\circ\pi\circ S^{-n}(z)}{h\circ\pi\circ S^{-n}(z)}\]
exists, $\Xi_h(f)\circ S = \Xi_h(f)$ and $\mathbb{E}_{\widehat{m}_h}(\Xi_h(f)|\cF_0) = f\circ \pi/h\circ \pi$. Moreover, for the signed invariant measure $\widehat{m}_f$, we have ${d\widehat{m}_f}/{d\widehat{m}_h}= \Xi_h(f)$. 
\end{proposition} 
\begin{proof} The proof that $({f\circ\pi\circ S^{-n}}/{h\circ\pi\circ S^{-n}}| \cF_n)$ is a bounded martingale and is the same as above and therefore omitted. Hence, $\Xi_h(f)$ is well defined and by bounded convergence, we have for $A \in \cB$ and $k \in \N$ that  
\begin{align*}
& \int_{S^k\pi^{-1}(A)} \Xi_h(f) d \widehat{m}_h \\
=& \lim_{n \to \infty} \int \1_A\circ\pi\circ S^{-k} \; \frac{f\circ\pi\circ S^{-n}}{h\circ\pi\circ S^{-n}} d \widehat{m}_h\\
=& \lim_{n \to \infty} \int \1_A\circ T^{n-k} f d\nu_\mathbf{o} = \int \1_A f d\nu_\mathbf{o} = m_f(A) = \widehat{m}_f(S^k\pi^{-1}(A)).
 \end{align*} 
Since $\cF$ is generated by $\{\cF_n\with n \in \N \}$, we have $\Xi_h(f) d\widehat{m}_h = d\widehat{m}_f$. The remaining assertion in the conditional expectation is a consequence of the above for $k=0$.  \qed
\end{proof}


\section{Applications and examples}

The construction of conformal measures has  the following application to conformal graph directed Markov systems. In order to have a zero of the pressure function, we have to assume that there exists $h>0$ such that \begin{align}
 \label{eq:zero_of_Gurevic} \limsup_{n \to \infty} \sqrt[n]{ \cL^n_{\varphi^h}(\1_{X_{\id}})(\xi,\id)} = 1, \\
\label{eq:multifractal} \| L_{\varphi^h}(\1) \|_\infty < \infty
\end{align}
are satisfied. It follows from standard arguments that the expression on the left hand side of (\ref{eq:zero_of_Gurevic}), seen as a function of $h$, is continuous and strictly decreasing to $0$ on its domain of definition. Hence, if there exists $h'$ such that the left hand side of (\ref{eq:zero_of_Gurevic}) is finite and greater than or equal to $1$ and (\ref{eq:multifractal}) holds, then there exists a zero of the pressure function. 
In the context of graph directed Markov systems, this property is known as strong regularity (see \cite{MauldinUrbanski:2003}).
Furthermore, if $|\cW^1|<\infty$, then this is true for $h'=0$, and in particular there always exists a zero of the pressure function in this case. 

Now let $\delta$ be given by (\ref{eq:zero_of_Gurevic}) and set $\rho_\delta := \exp(P_G(\te,\varphi^\delta))\geq 1$. It then follows from the Ruelle-Perron-Frobenius theorem for systems with the b.i.p. property (see, e.g., \cite{Sarig:2003a}) that there exists a $\rho_\delta/\varphi^\delta$-conformal probability measure $\mu_\delta$ and a Hölder continuous function $h_\delta$ with $L_{\varphi^\delta}(h_\delta) = \rho_\delta h_\delta$ such that $\te$ has the Gibbs-Markov property with respect to the invariant measure given by $h_\delta d\mu_\delta$. As an application of Theorem \ref{theo:conformal_general} and Proposition \ref{prop:measure} we obtain that there exists a $\sigma$-finite measure $\nu$ on $X$ which is $1/\varphi^\delta$-conformal, and which satisfies, for $w \in \cW^n$ and $x \in [w]$,
\begin{equation} \label{eq:conformal_relation}
 \nu([w,g]) \asymp  \Phi_{n}^\delta(x) \nu(X_{g\psi_n(x)}).
\end{equation} 

\begin{theorem} \label{theo:dimension_of_the_measure}
Assume that the group extension is symmetric and that property (C),  (\ref{eq:zero_of_Gurevic}) and  (\ref{eq:multifractal}) are satisfied. Then, for $\mu_\delta$-a.e. $(w_k) \in \Sig$,
\[ \lim_{n \to \infty} \frac{\log(\nu([w_1\cdots w_n,\id]))}{\log  \Phi_{n}(x)} = \delta  + \frac{P_G(\te,\varphi^\delta)}{\int (\log \varphi) h_\delta d\mu_\delta} .\] 
Moreover, the group $G$ is amenable if and only if the above limit is equal to $\delta$. If $G$ is non-amenable, then, for $\mu_\delta$-a.e. $(w_k) \in \Sig$,
\[ \lim_{n \to \infty} \rho_\delta^n \frac{\nu([w_1\cdots w_n,\id]))}{ (\Phi_{n}(x))^\delta} = 0.\] 

\end{theorem}
Before giving the proof, we sketch a straight forward application to conformal dynamical systems. Namely, if $\Sig$ is given by a conformal iterated function system, the inverse branch $\tau_w$ corresponds to a conformal map and $\Phi_{|w|}\circ \tau_w$ to its conformal derivative. In this situation, the above limit can be identified with the $\nu$-dimension $\dim_\nu$ of the support of $\mu_\delta$. Hence, with 
$H(h_\delta d \mu_\delta)$ referring to the entropy of $h_\delta d \mu_\delta$, it follows from the variational principle that 
\[ \dim_\nu(\supp (\mu_\delta)) = \delta  + \frac{P_G(\te,\varphi^\delta)}{\int (\log \varphi) h_\delta d\mu_\delta} 
= 2 \delta   + \frac{H(h_\delta d \mu_\delta)}{\int (\log \varphi) h_\delta d\mu_\delta}.  \]
Moreover, note that in many regular situations, $\delta$ is equal to the Hausdorff dimension $\dim(K)$ of the attractor $K$ of the iterated function system. 
In this situation, the amenability of $G$ is equivalent to $\dim_\nu(\supp (\mu_\delta))=\dim (K)$.  

\begin{proof}[of theorem \ref{theo:dimension_of_the_measure}] 
By symmetry and proposition \ref{prop:martingale_for_mu}, $\lim_n (\log \nu(X_{\psi_n(x)}))/n = \log \rho^{\delta}$. Hence, by (\ref{eq:conformal_relation}), 
\[ \lim_{n \to \infty} \frac{\log(\nu([w_n,\id]))}{\log  \Phi_{n}(x)} = \delta +  \lim_{n \to \infty}  \frac{\log \nu(X_{\psi_n(x)})}{\log  \Phi_{n}(x)} =  \delta + \frac{P_G(\te,\varphi^\delta)}{\lim_{n \to \infty} (\log \Phi_{n}(x))/n}.  
\]
The above limit exists by application of the ergodic theorem. The amenability criterion is an immediate corollary of  
Kesten's criterion for group extensions in \cite{Stadlbauer:2013}, where it is shown that $P_G(\te,\varphi^\delta)=0$ if and only if $G$ is amenable. For the remaining assertion, note that $\rho_\delta<1$ by non-amenability. The assertion then follows from  corollary \ref{cor:decay_of_the_measure_nu}. \qed 
\end{proof}

In order to have concrete examples of the $\sigma$-finite measure at hand, we give two examples from probability theory, where known local limit theorems give rise to  explicit expressions.

\begin{example} \label{ex:polya} The first example is Polya's random walk on $\Z^d$. Choose $( p_i \in (0,1) : i \in \{\pm 1,  \ldots, \pm d \})$ with $\sum_{i=1}^d (p_i + p_{-i}) =1$ and consider the random walk on $\Z^d$ with transition probabilities $P(\pm e_{i}) = p_{\pm i} $, where $e_{i}$ refers to the $i$-th element of the canonical basis of $\Z^d$. 

This random walk has an equivalent description through the following group extension. 
Let $\Sigma$ be the full shift with $2d$ symbols $\{-d,\ldots,-1,1,\ldots,d\}$ and $\varphi$ the locally constant function defined by $\varphi|_{[\pm i]}:= p_{\pm i}$. Note that $\sum_{i=1}^d (p_i + p_{-i}) =1$ implies that $L_\varphi(1)=1$. Moreover, it is well known that the measure defined by $\mu([i_1 \ldots i_n]):= p_{i_1} \cdots  p_{i_n}$ is $\te$-invariant, ergodic and $1/\varphi$-conformal. The associated group extension is defined through 
\[ \psi: \Sigma \to \Z^d, \; (i_1 i_2 \cdots) \mapsto \begin{cases} \phantom{-} e_{i_1} &: i_1>0 \\  -e_{-i_1} &: i_1<0 \end{cases}. \]
As  $\Sigma$ is the full shift and $\varphi$ is constant on cylinders, it follows from the construction that $\nu_{(x,g)}=\nu_{(y,g)}$ for all $x,y \in \Sigma$ and $g \in G$. Therefore, we only will write $\nu_g$ for $\nu_{(x,g)}$. In order to apply known local limit theorems from probability theory, observe that
\begin{align*}
 \cL_\varphi^n(\1_{X_{\id}})(x,g)  = \sum_{w \in \cW^n :\; \psi_n(w)= g} \phi_n(\tau_w(x)) 
 = P(X_n=g),
\end{align*}
where $X_n=h$ refers to the random walk at time $n$ started in the identity with distribution $(p_i)$ and $P$ to the probability of the associated Markov process. By the local limit theorem for Polya's random walk (\cite[theorem 13.12]{Woess:2000}), we have that, for $(k_1,\ldots,k_d) \in \Z^d$ and $n \in \N$ such that $n - (k_1+\cdots + k_d)$ is even,
\[ P(X_{n}=(k_1,\ldots,k_d)) \sim C n^{-d/2} \left({\textstyle 2 \sum_{i=1}^d \sqrt{p_i p_{-i}}}\right)^n  \prod_{i=1}^d \left( \sqrt{p_i /p_{-i}}\right)^{k_i}.\]
Hence, $\rho= 2 \sum_{i=1}^d \sqrt{p_i p_{-i}}$ and, with $\lambda_i := \sqrt{p_i /p_{-i}}$, 
\[  \cL_\varphi^n(\1_{X_{(k_1,\ldots,k_d)}})(x,\id) \sim C n^{-d/2} \rho^n \prod_{i=1}^d  \lambda_i^{-k_i}. \]
Recall that a random walk is called symmetric if $p_i=p_{-i}$ for all $i=1,\ldots,d$. 
The estimate then implies that $\rho =1$ if and only if the random walk is symmetric. Furthermore, by proposition \ref{for:ergodic_iff_divergence_type}, the term $n^{-d/2}$ implies that the group extension is ergodic and conservative with respect to $\nu$ if and only if $d=1$ or $d=2$. It is remarkable that this conclusion is independent of symmetry.  
In order to determine $\nu_{\id}$ explicitly, note that the local limit theorem implies that
\begin{align*}
\nu_{\id}(X_{(k_1,\ldots,k_d)}) = \lim_{k \to \infty}
  \frac{\sum_{n\in \N} {b_{n}} s_k^{-n} (\cL_{\varphi}^n \1_{X_{(k_1,\ldots,k_d)}})(x,\id)}
  {\sum_{n\in \N} {b_{n}} s_k^{-n} (\cL_{\varphi}^n \1_{X_{\id}})(x,\id)} = \prod_{i=1}^d  \lambda_i^{-k_i}.
\end{align*}
Using conformality then gives that, for a cylinder $[(i_1,\ldots,i_n),z]$ in $\Sigma \times \Z^d$,
\begin{align}
\nonumber \nu_{\id}([(i_1\ldots i_n),z]) & = \rho^{-n} p_{i_1} \cdots  p_{i_n} \nu_{\id}( X_{z + \psi_n(i_1 \ldots i_n)})\\
\nonumber  & = \rho^{-n}  p_{i_1} \cdots  p_{i_n}  \nu_{\id}( X_{z}) 
\prod_{k=1}^n  \lambda_{i_k}^{-1}  =  \rho^{-n} \nu_{\id}( X_{z})  \prod_{k=1}^n  \sqrt{p_{i_k}p_{-i_k}} \\
\label{eq:measure_formula_Zd} &= \frac{1}{2^n} \nu_{\id}(X_{z})  \prod_{k=1}^n \frac{\sqrt{p_{i_k}p_{-i_k}}}{\sum_{i=1}^d \sqrt{p_i p_{-i}}}
\end{align}
In particular, the last term in (\ref{eq:measure_formula_Zd}) reveals the local symmetry
\[\nu_{\id}([(i_1\ldots i_k \ldots i_n),z]) = \nu_{\id}([(i_1\ldots -i_k \ldots i_n),z]), \quad (k \in {1,\ldots , n}), \]
whereas globally, the measure is multiplicative with respect to the last component, that is 
\[
\nu_{\id}([(i_1\ldots i_n),z_1 +  z_2]) = \nu_{\id}([(i_1 \ldots i_n),z_1]) \nu_{\id}([(i_1 \ldots i_n),z_2]).
\]
Furthermore, (\ref{eq:measure_formula_Zd}) implies that the the function $\mathbb{K}$ from Theorem \ref{theo:family_of_measures-2} is given by
\[\mathbb{K}(\delta_{(x,g)},(y,h)) = \frac{d\nu_g}{d\nu_{\id}}(y,h) = \nu(X_g).\]
These considerations might be summarized as follows. If $\varphi$ is symmetric, then $\rho =1$ and $\nu(X_g)=1$ for all $g \in \Z^d$.
If $\varphi$ is not symmetric, then $\rho < 1$ and $\{\nu_{\id}(X_g) : g \in \Z^d\}$ neither is bounded from below nor from above. Moreover, the function $h$ defined by $h(x,g) := \nu_g(X_{\id})$ is an  $\cL_\varphi$-proper function by remark \ref{rem:eigenfunctions}. Therefore,  $dm := h d\nu$ is  $T$-invariant. However, as it easily can be verified, $m(X_g)=1$ for all $g \in \Z^d$ and, in particular, $m$ is the measure associated to the symmetric random walk with transition probabilities $P(\pm e_{i}) = \sqrt{p_i p_{-i}}/(2 \sum_k \sqrt{p_k p_{-k}})$. 
\end{example}

\begin{example} In this example, we replace the group $\Z^d$ with the free group $\F_d$ with $d$ generators $g_1,\ldots,g_d$.
As above, the transition probabilities are given by $P(g_{\pm i}) = p_{\pm i}$, where $g_{-i}:=g_i^{-1}$. The construction of the associated group extension then has to be adapted only be changing $\psi$ to 
\[ \psi: \Sigma \to \F_d, \; (i_1 i_2 \cdots) \mapsto g_{i_1}.  \]
As above, we now apply a local limit theorem. The result of Gerl and Woess in \cite{GerlWoess:1986} is applicable in full generality, however, for ease of exposition, we restrict ourselves to the special case where $q:= \sqrt{p_i p_{-i}}$ does not depend on $i$. 
Then, by (5.3) and (5.4) in \cite{GerlWoess:1986}, we have that $\rho = 2q \sqrt{2d-1}$ and that 
\begin{eqnarray} \label{eq:measure_F_d} \lim_{n \to \infty} \frac{P(X_n = g_{i_1} \cdots g_{i_k})}{P( X_n = \id)} = \left( 1 + {\textstyle \frac{d-1}{d}} k \right) \left({2d-1}\right)^{-k/2} \prod_{i=1}^k  \lambda_{i_k},
\end{eqnarray}
for $n$ and $k$ even and $g_{i_1} \cdots g_{i_k}$ in reduced form, that is ${i_l} \neq -i_{l+1}$, for $l=1,\ldots, n-1$. 
Also note that there is a misprint in equation (5.4) in \cite{GerlWoess:1986}. In there, one has to replace $d/(d-1)$ in the first factor by its inverse as in (\ref{eq:measure_F_d}). As above, the right hand side in (\ref{eq:measure_F_d}) is equal to $\nu_{\id}(X_{g_{i_1} \cdots g_{i_k}})$. Using the identities for $q$ and $\rho$ and setting $C_k := 1 + k (d-1)/d$, this gives that
\begin{align*}
\nu_{\id}(X_{g_{i_1} \cdots g_{i_k}}) = 
 C_k {(2/\rho)^k}   \prod_{i=1}^k  q\lambda_{-i_k} =  
 C_k {(2/\rho)^k}   \prod_{i=1}^k  p_{-i_k}.    
\end{align*}
Since the identity requires that $g = g_{i_1} \cdots g_{i_k}$ is in reduced form, we have to introduce the following operations on finite words in order to obtain a formula for arbitrary cylinders. For $w=(i_1 \ldots i_n) \in \cW^n$, there exists a unique $k \le n$ and a word $(j_1\ldots j_k)\in \cW^k $ such that $\psi_n(w)= g_{j_1} \cdots g_{j_k}$ is in reduced form. We will refer to 
 $\mathfrak{r}(w):= (j_1\ldots j_k)$ as the \emph{active part} of $w$, whereas the word which is obtained by deleting the entries of $\mathfrak{r}(w)$ from $w$ is referred to as the \emph{inactive part} $\mathfrak{i}(w) \in \cW^{n-k}$ of $w$. Note that $\psi_k (\mathfrak{r}(w)) = \psi_n(w)$ and $\psi_{n-k} (\mathfrak{i}(w)) = \id$. Moreover, for a given word $v= (i_1\ldots i_n) \in \cW^n$, we will refer to $\kappa(v):= (-i_n,\ldots, -i_2, -i_1)$ as the \emph{inverse word} of $v$. For ease of notation, we also will make use of the Bernoulli measure on $\Sigma$ defined through $\mu([i_1\ldots i_n]) = p_{i_1} \cdots p_{i_n}$. 

As it will be shown below, the measure of a cylinder $[w,g]$, for $w \in \cW^n$ and $g\in G$ and the function $\mathbb{K}$ given by theorem \ref{theo:family_of_measures-2} depend on possible cancelations of the concatenation of the path to $g\in G$ and $w$. So, let $v_g \in \cW^m$ be given by $\psi_m(v_g)=g$ and $\mathfrak{i}(v_g)=\emptyset$, that is $v_g$ is given by the reduced form of $g$. With $k:= |\mathfrak{r}(v_g w)|$, the conformality of $\nu_{\id}$ implies that 
\begin{align*} 
\nu_{\id}([w,g])  = \rho^{-n} \mu([w])  \nu_{\id}(X_{g\psi_n(w)})
 = \rho^{-n} \mu([w])  C_k (2/\rho)^k  \mu([\kappa(\mathfrak{r}(v_g w))]) 
\end{align*}
in case that $\mathfrak{r}(v_g w) \neq \emptyset$. If $\mathfrak{r}(v_g w) \neq \emptyset$, then 
 $\nu_{\id}([w,g]) = \rho^{-n} \mu([w])$ by the same arguments. The identity now allows to determine the function $\mathbb{K}$  explicitly. That is, for $g_1,g_2 \in G$, $x \in \Sig$ and $(w_{(n)})$ with $w_{(n)} \in \cW^n$ and $x = \lim_n [w_{(n)}]$, we have $\nu_{\id}$ a.s., that
\begin{align*} 
 & \mathbb{K}(g_1, (x,g_2)) 
= \lim_{n \to \infty}  \frac{\nu_{g_1}([w_{(n)},g_2])}{\nu_{\id}([w_{(n)},g_2])} 
  = \lim_{n \to \infty} \frac{\nu_{\id}([w_{(n)}, g_1^{-1}g_2])}{\nu_{\id}([w_{(n)},g_2])}  \\
\\ 
=& \lim_{n \to \infty} 
\frac{ C_{|\mathfrak{r}(v_{g_1^{-1}g_2} w_{(n)})|}}{C_{|\mathfrak{r}(v_{g_2} w)|}} 
 \left(\frac{2}{\rho}\right)^{|\mathfrak{r}(v_{g_1^{-1}g_2} w_{(n)})| - |\mathfrak{r}(v_{g_2} w_{(n)})|}  
 \frac{\mu([\kappa(\mathfrak{r}(v_{g_1^{-1}g_2} w_{(n)}))])}{\mu([\kappa(\mathfrak{r}(v_{g_2} w_{(n)}))])}. 
%
%
%
\end{align*} 
Observe that total dissipativity implies that $(\psi_n(x):n \in \N) = (\psi_n(w_{(n)}): n \in \N)$ will almost surely only return finitely many times to a finite subset of $G$. Hence, the first term in the product converges to 1 whereas the second and third eventually are constant. By setting  ${k}_{g_1,g_2}(x) := \lim_{n \to \infty} |g_1^{-1}g_2 \psi_n(x)| - |g_2 \psi_n(x)|$, analyzing the cancelations in $v_{g_1^{-1}}v_{g_2} w_{(n)}$ and $v_{g_2} w_{(n)}$ and using that $q^2=p_ip_{-i}$, it follows that 
\begin{align*}
  & \mathbb{K}(g_1, (x,g_2))  
  = (2/\rho)^{{k}_{g_1,g_2}(x)} \lim_{n\to \infty}\frac{\mu([\kappa(\mathfrak{r}(v_{g_1^{-1}g_2} w_{(n)}))])}{\mu([\kappa(\mathfrak{r}(v_{g_2} w_{(n)}))])}  \\
 =& (2/\rho)^{{k}_{g_1,g_2}(x)}   \cdot \frac{\mu([v_{g_1}])}{q^{|g_1|-{k}_{g_1,g_2}(x)}} 
 = (2d-1)^{-\frac{{k}_{g_1,g_2}(x)}{2}} \sqrt{\mu([v_{g_1}])/\mu([\kappa v_{g_1}])}.
\end{align*}
The regularity of $h$ can now be analyzed through ${k}_{g_1,g_2}$. In order to do so, for each open subset $U$ of $X$ and $g =g_1 \in G$, observe that ${k}_{g,{g_2}}(U)$ is equal to $\{-|g|,2 -|g| ,\ldots,|g| -2, |g|\}$. This implies that
 \[\sup_{z,\tilde{z} \in U} \frac{\mathbb{K}(g, z)}{\mathbb{K}(g, \tilde{z})} = (2d-1)^{|g|} \sqrt{\frac{\mu([v_{g}])}{\mu([\kappa v_{g}])}}.\]
In particular, the fluctuations of $h(g, \cdot)$ only depend on $|g|$ and the quantity $\mu([v_{g}])/\mu([\kappa v_{g}])$, which measures the asymmetry of the random walk. 
If the random walk is symmetric, that is $p_i = p_{-i} = 1/2d$ for all $i= 1, \ldots, d$, then this simplifies to 
\begin{align*}
\rho = \sqrt{2d -1}/d,\; \nu_{\id}(X_g) =  C_{|g|} (2d -1)^{-\frac{|g|}{2}},\; \mathbb{K}(g_1, (x,g_2)) = (2d -1)^{-\frac{{k}_{g_1,g_2}(x)}{2}}. 
\end{align*}
\end{example}

\subsection*{Acknowledgements}
The author acknowledges support by CNPq through PQ 310883/2015-6 and Projeto Universal 426814/2016-9.


\end{document}